\documentclass[12pt]{amsart}

\theoremstyle{definition}
\theoremstyle{plain}

\usepackage[left=2.8cm,top=2.8cm,right=2.8cm,bottom = 2.4cm]{geometry}
\usepackage{amsfonts}  
\usepackage{amsmath}
\usepackage{amscd}
\usepackage{amsrefs}
\usepackage{amssymb} 
\usepackage{amsthm} 
\usepackage{bbm}
\usepackage{bm}
\usepackage{enumerate}
\usepackage{enumitem}
\usepackage{hyperref}
\usepackage{latexsym}
\usepackage{mathdots}
\usepackage{mathtools}
\usepackage{xfrac}
\usepackage{array}

\newtheorem{lemma}{Lemma}[section]

\newtheorem{remark}[lemma]{Remark}
\newtheorem{theorem}[lemma]{Theorem}
\newtheorem{corollary}[lemma]{Corollary}
\theoremstyle{definition}

\newcommand{\CDOT}{\vcenter{\hbox{\scalebox{2}{$\cdot$}}}}

\newcommand{\C}{\mathbb{C}}
\newcommand{\Chat}{\widehat{\C}}

\newcommand{\FWC}{\mathfrak{C}}
\newcommand{\FWCbar}{\overline{\mathfrak{C}}}
\newcommand{\R}{\mathbb{R}}

\DeclareMathOperator{\ad}{ad}

\DeclareMathOperator{\End}{End}
\DeclareMathOperator{\id}{id}
\DeclareMathOperator{\Res}{Res}
\DeclareMathOperator{\sgn}{sgn}

\DeclareMathOperator{\Tr}{Tr}
\DeclareMathOperator{\Z}{\mathbb{Z}}
\DeclareSymbolFont{bbold}{U}{bbold}{m}{n}
\DeclareSymbolFontAlphabet{\mathbbold}{bbold}
\DeclareMathSymbol{\shortminus}{\mathbin}{AMSa}{"39}

\title{Generating Functions for Casimir Invariants of Simple Lie Algebras}
\begin{document}
	\author[Michael Flattery]{Michael Flattery$^{\dagger}$}
\thanks{$^{\dagger}$Supported by a University of Galway College of Science and Engineering Scholarship}
%\address{School of Mathematical and Statistical Sciences\\ 	University of Galway, Galway H91 TK33, Ireland}
\email{flatterymichael@gmail.com}
	\author{Michael P. Tuite}
%	\thanks{$^{*}$Corresponding author.}
	\email{michael.tuite@universityofgalway.ie}
	\address{School of Mathematical and Statistical Sciences\\ 
		University of Galway, Galway H91 TK33, Ireland}

\begin{abstract}
	We consider generalised Casimir operators associated with irreducible representations for finite-dimensional, semisimple, complex Lie algebras. Using Klimyk's formula and Okubo's formula, we derive a rational ordinary generating function for the Casimir operator eigenvalues for simple Lie algebras which does not rely on knowledge of the Weyl group. We also directly reproduce generating function expressions for the classical Lie algebras originally due to Perelomov and Popov.
\end{abstract}

\maketitle

\section{Introduction}
Let $\mathfrak{g}$ be a rank $n$ finite-dimensional semisimple complex Lie algebra and $U(\mathfrak{g})$ the corresponding  universal enveloping algebra  e.g. \cites{Hu1, Kn}. 
%We can construct the which inherits algebraic information about $\mathfrak{g}$ \cites{Hu1, Kn}. 
The central elements $Z(U(\mathfrak{g}))$ are determined by $n$ algebraically independent invariants \cites{GOR, Hu2, Ra}. This paper specifically focuses on the invariants given by generalised Casimir operators associated with a given $\mathfrak{g}$-irreducible representation  \cite{Ok}. For $\mathfrak{g}$ simple, it is known that $Z(U(\mathfrak{g}))$ is generated by $n$ specific Casimir operators associated with the minimal irreducible representation of $\mathfrak{g}$ except for Lie algebras of type $D_{n}$ where at least one of the algebraically independent set must be associated with a half-spinor representation. In this paper we describe a novel approach to calculating the eigenvalue of a generalised Casimir operator acting on an irreducible representation. We utilise Klimyk's formula \cites{Kl, Sn} which determines the irreducible decomposition
of the tensor product of two irreducible representations and Okubo's formula \cite{Ok} for Casimir eigenvalues which relies on such a decomposition. Employing an extension of the Weyl dimension formula to any integral weight, we derive an ordinary rational generating formula for Casimir eigenvalues which does not require knowledge of the Weyl group. We directly derive closed form generating function expressions for the classical Lie algebras, originally due to Perelomov and Popov \cite{PP}.

In Section~\ref{sec:Prelim} we begin with some definitions and results concerning the Lie algebra $\mathfrak{g}$. We discuss both the standard and dot (or affine) action of the Weyl group on the dual Cartan space $\mathfrak{h}^{*}$. We define the generalised integer index $p\in \Z_{\ge 0}$ Casimir operator, denoted by $K_{p}^{\kappa}$, associated with an irreducible representation $V(\kappa)$  for a given integral highest weight vector $\kappa$. We also specify sets of generalised Casimir operators which generate $Z(U(\mathfrak{g}))$ for simple $\mathfrak{g}$. We also introduce Dynkin indices to explicitly describe the Casimir  operators with respect to our choice of normalisation. 

In Section~\ref{sec:KlimOk} we consider the Okubo and Klimyk formulas. By Schur's Lemma, any element of $Z(U(\mathfrak{g}))$ acts on an irreducible representation as scalar multiplication by an eigenvalue.  Okubo's formula \cite{Ok} provides a means of calculating the eigenvalue, we denote by $c_{p}^{\kappa}(\lambda)$, of the action of the generalised Casimir operator $K_{p}^{\kappa}$ on an irreducible representation $V(\lambda)$ for a given integral highest weight vector $\lambda$.  
Okubo's  formula requires knowledge of the irreducible decomposition of $V(\lambda)\otimes V(\kappa)$. Klimyk's formula \cites{Kl, Sn}, which employs the Weyl group dot  action, gives us one method for determining this  decomposition.
We also discuss the properties of the eigenvalue $\xi_{\lambda,\kappa}(\nu)$ of an operator $Q_{\lambda,\kappa}$ acting on an irreducible representation $V(\nu)$ that arises in proving Okubo's formula. This eigenvalue  proves to have an essential role in Section~\ref{sec:GenFun}.

In Section~\ref{sec:GenFun} we show  how to compute the eigenvalue  $c_{p}^{\kappa}(\lambda)$ for the $K_{p}^{\kappa}$ action on $V(\lambda)$ based only on knowledge of the positive roots of $\mathfrak{g}$, $V(\kappa)$ weights and  the highest weight $\lambda$. 
We begin the section with a discussion of  natural extensions of the Weyl dimension formula to a function $D(\tau)$ and the $Q_{\lambda,\kappa}$ eigenvalues to  $\xi_{\lambda,\kappa}(\tau)$ for all $\tau\in\mathfrak{h}^{*}$. 
Our main result, Theorem~\ref{theor:CasGen}, gives a rational expression for the generating function $\sum_{p\ge 0}c_{p}^{\kappa}(\lambda)z^{p}$ expressed as a sum over $\mu\in V(\kappa)$ where the summand depends on $D(\lambda+\mu)$ and $\xi_{\lambda,\kappa}(\lambda+\mu)$. This bypasses a significant computational barrier in the Okubo formula calculation of Casimir eigenvalues since we are not required to use the Weyl group or the irreducible decomposition of $V(\lambda)\otimes V(\kappa)$. We directly derive closed form generating function expressions, originally due to Perelomov and Popov \cite{PP}, for the classical Lie algebras for eigenvalues for the generalised Casimir operators associated with the first fundamental irreducible representation acting on $V(\lambda)$.

Section~\ref{sec:Conc} concludes with some general remarks.

\section{Preliminaries}\label{sec:Prelim}
We begin with a review of the background and conventions used in this paper to avoid ambiguity. We freely use the definitions of Humphreys \cites{Hu1,Hu2} and of Knapp \cite{Kn}, though some distinct notation and shorthand expressions are introduced. Any further background from other sources will be specified as such. Throughout, we assume $\mathfrak{g}$ is a semisimple  Lie algebra over $\C$ of finite dimension $d$ and rank $n$. Where we require simplicity it is explicitly noted.

Let $\Delta$ be a system of roots for $\mathfrak{g}$ relative to the Cartan subalgebra $\mathfrak{h}$ and then let $\Delta^{+}$ be a choice of positive roots and $\alpha_{1},\ldots,\alpha_{n}$ be the induced simple roots for a given choice of labeling. Let $(\cdot,\cdot)$ denote a normalised Killing form on $\mathfrak{g}$ and, for $\mu\in\mathfrak{h}^{*}$, let $H_{\mu}$ be the unique Cartan element such that we have $(H_{\mu},H)=\mu(H)$ for any $H\in\mathfrak{h}$. Define a symmetric bilinear form $\langle\mu,\nu\rangle:=(H_{\mu},H_{\nu})$ for any $\mu,\nu\in\mathfrak{h}^{*}$. Define $\Vert\mu\Vert:=\sqrt{\langle\mu,\mu\rangle}$ for any choice of $\mu\in\mathfrak{h}^{*}$ and note that the normalisation of $(\cdot,\cdot)$ is chosen so that  $\Vert\alpha\Vert=\sqrt{2}$ if $\alpha\in\Delta$ is a long root.

Let $\varpi_{1},\ldots,\varpi_{n}$ be the fundamental weights relative to $\alpha_{1},\ldots,\alpha_{n}$. For $\mathfrak{g}$ simple, we choose our ordering so that $\varpi_{1}$ is the highest weight of a minimal dimension irreducible representation. We define $\mu\in\mathfrak{h}^{*}$ to be dominant if $\langle\mu,\varpi_{i}\rangle\geq0$ for each $1\leq i\leq n$, that is, if $\mu$ has non-negative inner product with each fundamental weight. Following Humphreys \cite{Hu1}, we will call $\mu$ strongly dominant (also referred to as regular) if each of these inequalities are instead strict inequalities. We also introduce the term weakly dominant to refer to any such $\mu$ that is dominant but not strongly dominant, that is, $\mu$ is weakly dominant if it is dominant and $\langle\mu,\varpi_{i}\rangle=0$ for at least one choice of $1\leq i\leq n$. Let $\FWC:=\mathfrak{C}(\Delta^{+})$ denote the fundamental Weyl chamber consisting of the subset of all strongly dominant elements of $\mathfrak{h}^{*}$. Note that $\FWC$ is an open set with closure $\FWCbar$ consisting of all dominant elements and with boundary $\partial\FWC$ consisting of the weakly dominant elements.

 Define the Weyl vector as $\rho:=\sum_{i=1}^{n}\varpi_{i}$, or equivalently, $\rho=\tfrac{1}{2}\sum_{\alpha\in\Delta^{+}}\alpha$. 
We define the set of integral elements as $\Lambda:=\left\{\sum_{i=1}^{n}k_{i}\varpi_{i}\vert k_{1},\ldots,k_{n}\in\Z\right\}\subset \mathfrak{h}^{*}$ and denote the set of dominant integral elements as $\Lambda^{+}$, that is, $\Lambda^{+}:=\FWCbar\cap\Lambda$. The sets of strongly dominant integral elements  and weakly dominant elements in $\mathfrak{h}^{*}$ can be given explicitly as, respectively, $\Lambda^{+}+\rho=\lbrace\lambda+\rho\,\vert\,\lambda\in\Lambda^{+}\rbrace\subset\FWC$ and $\Lambda^{+}\backslash(\Lambda^{+}+\rho)\subset\partial\FWC$.

For $\mathfrak{g}$ simple, we embed $\mathfrak{h}^{*}$ in $\R^{k}$, for positive integer $k\geq n$, equipped with a Euclidean inner product $u\cdot v$, for $u,v\in\R^{k}$, chosen so that $\mu\cdot\nu = r\langle\mu,\nu\rangle$ for constant $r\in\Z^{+}$ for all choices of $\mu,\nu\in\mathfrak{h}^{*}\hookrightarrow\R^{k}$. We also define $u^{2}:=u\cdot u$ to be the Euclidean squared norm for any choice of $u\in\R^{k}$, so we then have $\mu^{2}=r\Vert\mu\Vert^{2}$ for all $\mu\in\mathfrak{h}^{*}$. For later brevity, we adopt the notation $\mathbbold{1}_{k}:=e_{1}+\cdots+e_{k}$ in $\R^{k}$ where $e_{i}$ is the $i^{\text{th}}$ standard Euclidean basis vector. We choose the embedding of simple Lie algebras as shown in Table~\ref{tab:EucEmbed}. 

\begin{table}[ht!]
	{\setlength{\extrarowheight}{4pt}
		\begin{tabular}%{|c|c|c|c|}
			{|@{\hspace{0.82em}}c@{\hspace{0.82em}}|@{\hspace{0.82em}}l@{\hspace{1.52em}}|@{\hspace{0.82em}}c@{\hspace{0.82em}}|@{\hspace{0.82em}}c@{\hspace{0.82em}}|}\hline
			Type & Simple Roots $\alpha_{1},\ldots,\alpha_{n}$ & $r$ & $k$\\ \hline\hline
			$A_{n}$ & $\alpha_{i} = e_{i}-e_{i+1}, \hspace{1em} 1\leq i\leq n,\;n\ge 1$ & 1 & $n+1$ \\ \hline
			$B_{n}$ & $\alpha_{i} = e_{i}-e_{i+1},  \;\alpha_{n}=e_{n},\hspace{1em}1\leq i\leq n-1, \;  n\ge 2$ & 1 & $n$\\ \hline
			$C_{n}$ & $\alpha_{i} = e_{i}-e_{i+1},\; \alpha_{n}=2e_{n},\hspace{1em} 1\leq i\leq n-1, \;n\ge 3$ & 2 & $n$\\ \hline
			$D_{n}$ & $\alpha_{i} = e_{i}-e_{i+1},\; \alpha_{n}=e_{n\shortminus1}+e_{n},\hspace{1em} 1\leq i\leq n-1, \; n\ge 4$ & 1 & $n$\\ \hline
			$E_{6}$ & as with $D_{5}$ with the addition of $\alpha_{6} = -\tfrac{1}{2}\mathbbold{1}_{8}$ & 1 & 8\\ \hline
			$E_{7}$ & as with $D_{6}$ with the addition of $\alpha_{7} = -\tfrac{1}{2}\mathbbold{1}_{8}$ & 1 & 8\\ \hline
			$E_{8}$ & as with $D_{7}$ with the addition of $\alpha_{8} = -\tfrac{1}{2}\mathbbold{1}_{8}$ & 1 & 8\\ \hline
			$F_{4}$ & as with $C_{3}$ with the addition of $\alpha_{4} = -\mathbbold{1}_{4}$ & 2 & 4\\ \hline
			$G_{2}$ & $\alpha_{1} = e_{1}-e_{2}, \; \alpha_{2} = 3e_{2}-\mathbbold{1}_{3}$ & 3 & 3\\ \hline
\end{tabular}}
\caption{Embedding and Labeling of Simple Roots in $\R^{k}$.}
\label{tab:EucEmbed}
\end{table}

Let $W$ be the Weyl group of $\Delta$, generated by the $n$ simple Weyl reflections as in the presentation $W=\left\langle s_{1},\ldots,s_{n}\mid s_{i}^{2}=(s_{i}s_{j})^{k_{i,j}}=1,i\neq j\right\rangle$ where
\begin{align*}
	k_{i,j}:=\left(\tfrac{1}{\pi}\arccos\left(\dfrac{\langle\alpha_{i},\alpha_{j}\rangle}{\Vert\alpha_{i}\Vert\,\Vert\alpha_{j}\Vert}\right)\right)^{-1}\hspace{2em}\text{for }i\neq j.
\end{align*}
We define two actions of $W$ on $\mu\in\mathfrak{h}^{*}$ given for each generator $s_{i}$  by:
\begin{align*}
	s_{i}(\mu)&:=\mu-2\dfrac{\langle\alpha_{i},\mu\rangle}{\langle\alpha_{i},\alpha_{i}\rangle}\alpha_{i},\hspace{2.5em}\text{The Standard Action,}%\hspace{1em}\text{and}
	\\
	s_{i}\CDOT\mu&:=s_{i}(\mu+\rho)-\rho,\hspace{3.42em}
	\text{The Dot (or Affine) Action}.
\end{align*}
It is easy to check that $w\CDOT\mu=w(\mu+\rho)-\rho$ for all $w\in W$.
We say $\mu,\nu\in\mathfrak{h}^{*}$ are linked, denoted $\mu {\sim}\nu$, with linkage $w\in W$ when $\mu=w\CDOT\nu$.
%
%We say $\mu,\nu\in\mathfrak{h}^{*}$ are linked, denoted $\mu\sim\nu$, if $\mu\in W\boldsymbol{\cdot}\nu$ and call $w\in W$ the linkage between $\mu$ and $\nu$ if $\mu=w\boldsymbol{\cdot}\nu$.
% 
This linkage is unique iff $\mu+\rho$ lies in a (open) Weyl chamber iff $\mu\sim\tau$ for some $\tau$ such that $\tau+\rho\in\FWC$ is strongly dominant, that is, for some choice of $\tau\in\FWC-\rho = \left\{\lambda-\rho\hspace{0.5em}\vert\hspace{0.5em}\lambda\in\FWC\right\}$, since the standard action of $W$ on the set of Weyl chambers is transitive.

Define the standard signature $\sgn(w)=(-1)^{s}$ for  $w\in W$, a product of $s$ simple Weyl reflections. We also define a signature function for $\mu\in\mathfrak{h}^{*}$ following \cite{Sn} as:
\begin{align}\label{eq:sgn_hstar}
	\sgn(\mu):=\begin{cases}
		\sgn(w), \hspace{0.5em}\mbox{ when } \mu\sim\nu \mbox{ with unique linkage }w \mbox{ for some }\nu+\rho\in\FWC ,
		\\
		0, \hspace{3em}\mbox{ otherwise}.
	\end{cases}
\end{align}
\\
\begin{remark}\label{rem:SignatureRemarks}
	\begin{enumerate}\leavevmode
		\item[(i)] 
		Note that $\sgn(\mu)\neq0$ iff $\mu+\rho$ lies inside of a Weyl chamber iff $\mu+\rho$ lies in the standard Weyl orbit of a strongly dominant element iff $\mu+\rho$ is linked to an element of $\FWC-\rho$. Conversely, $\sgn(\mu)=0$ iff $\mu+\rho$ lies in the boundary between Weyl chambers iff $\mu+\rho$ is in the standard Weyl orbit of a weakly dominant element iff $\mu$ is linked to an element of $\partial\FWC-\rho$.
		\item[(ii)]  
		For $\mu\in\Lambda$ in particular, since both actions of $W$ preserve the integral property, we have $\sgn(\mu)\neq0$ iff $\mu$ is linked to a dominant element as $\Lambda\cap(\FWC-\rho)=\Lambda^{+}$.
	\end{enumerate}
\end{remark}
Any finite-dimensional representation $V$, with representation map $\sigma_{V}:\mathfrak{g}\rightarrow\End(V)$, is completely reducible with a decomposition that is unique up to reordering. Also, two weights $\mu,\nu\in\mathfrak{h}^{*}$ in the same standard Weyl orbit have the same multiplicity in $V$. We note $V$ extends to a representation of $U(\mathfrak{g})$, the universal enveloping algebra of $\mathfrak{g}$, via the map $\sigma_{V}':U(\mathfrak{g})\rightarrow \End(V)$ given explicitly by $\sigma_{V}'(Y_{1}\otimes\cdots\otimes Y_{l}):=\sigma_{V}(Y_{1})\circ\cdots\circ\sigma_{V}(Y_{l})$ for any
$Y_{1},\ldots, Y_{l}\in\mathfrak{g}$. By an abuse of notation, we also refer to $\sigma_{V}'$ by the notation $\sigma_{V}$.  

For $\lambda\in\Lambda^{+}$, we denote by $V(\lambda)$ the irreducible representation of highest weight $\lambda$ with representation map $\sigma_{\lambda}:\mathfrak{g}\rightarrow\End(V(\lambda))$. We denote by $\Lambda(\lambda)$ the subset of $\Lambda$ of weights appearing in $V(\lambda)$, denote by $m_{\lambda}(\mu)$ the multiplicity of $\mu\in\Lambda$ in the representation $V(\lambda)$ and denote by $n_{V}(\lambda)$ the multiplicity of $V(\lambda)$ in the decomposition of a representation $V$.

The Weyl dimension formula gives $\dim V(\lambda)$ for $\lambda\in\Lambda^{+}$ as a closed expression, given by
\begin{align}\label{DimensionFormula}
	\dim V(\lambda) = \prod_{\alpha\in\Delta^{+}}\dfrac{\langle\lambda+\rho,\alpha\rangle}{\langle\rho,\alpha\rangle}.
\end{align}
Let $\lbrace X_{i}\rbrace_{i=1}^{d}$ be a basis of $\mathfrak{g}$ and  $\lbrace\widetilde{X}_{i}\rbrace_{i=1}^{d}$ the dual basis with respect to the normalised  form $(\cdot,\cdot)$. We define the standard quadratic Casimir operator to be the element of $U(\mathfrak{g})$ given by:
\begin{align}\label{StandardQuadraticCasimir}
	K_{2}:=\sum_{i,j=1}^{d}\left(\widetilde{X}_{i},\widetilde{X}_{j}\right)X_{i}\otimes X_{j}.
\end{align}
It is known that $K_{2}$ is basis-independent and central in $U(\mathfrak{g})$ and so, by Schur's Lemma, the endomorphism $\sigma_{\lambda}(K_{2})$ acts on $V(\lambda)$ by an eigenvalue $c_{2}(\lambda)$ for $\lambda\in\Lambda^{+}$. It is well known that  $c_{2}(\lambda)=\langle\lambda,\lambda+2\rho\rangle$.

Let $p\in \Z^{+}$ and $\kappa\in\Lambda^{+}$. We define a set of multilinear $p$-forms on $\mathfrak{g}$ by
\begin{align}\label{MultilinearForm}
	\left(Y_{1},\ldots,Y_{p}\right)_{\kappa}:=\Tr_{\hspace{0.1em}V(\kappa)}\left(\sigma_{\kappa}\left(Y_{1}\otimes\cdots\otimes Y_{p}\right)\right),
\end{align}
for $Y_{1},\ldots,Y_{p}\in\mathfrak{g}$. For $p=2$, we have $(\cdot,\cdot)_{\kappa}$ defines an invariant, symmetric, bilinear form on $\mathfrak{g}$. Let $\lbrace\widetilde{X}_{i}^{\kappa}\rbrace_{i=1}^{d}$ be the dual basis of $\lbrace X_{i}\rbrace_{i=1}^{d}$ with respect to the bilinear form $(\cdot,\cdot)_{\kappa}$. Following Okubo \cite{Ok}, we define the generalised degree $p\in\Z_{\geq0}$ Casimir operator associated with the representation $V(\kappa)$  by 
\begin{align}\label{GeneralisedCasimir}
	K_{p}^{\kappa}:=
	\begin{cases}
		\dim V(\kappa)\id_{U(\mathfrak{g})}  \mbox{ for } p=0,
		\\
		\sum_{i_{1},\ldots,i_{p}=1}^{d}\left(\widetilde{X}_{i_{1}}^{\kappa},\ldots,\widetilde{X}_{i_{p}}^{\kappa}\right)_{\kappa}X_{i_{1}}\otimes\cdots\otimes X_{i_{p}} \mbox{ for } p\geq 1.
	\end{cases}
\end{align}
$K_{p}^{\kappa}$ is basis-independent and lies in the centre $Z(U(\mathfrak{g}))$ of $U(\mathfrak{g})$. 
In particular, $K_{1}^{\kappa}$ is in the centre of $\mathfrak{g}$ and hence $K_{1}^{\kappa}=0$ since $\mathfrak{g}$ is semisimple. More generally,  $\sigma_{\lambda}(K_{p}^{\kappa})$ acts by an eigenvalue $c_{p}^{\kappa}(\lambda)$ on $V(\lambda)$,  by Schur's Lemma.  

It is known that $n$ distinguished generalised Casimir operators generate  $Z(U(\mathfrak{g}))$ \cites{GOR, Hu2, Ra}.  Table~\ref{tab:ZU_Gens} contains a list of these $n$ generators for the simple Lie algebras \cites{Be, Be2, GOR}. Note that $\varpi_{n-1}$ and $\varpi_{n}$ denote the half-spinor fundamental representations for $D_{n}$.
\begin{table}[h!]
		{\setlength{\extrarowheight}{3pt}\begin{tabular}{|@{\hspace{1em}}c@{\hspace{1em}}|@{\hspace{1em}}l@{\hspace{4em}}|}\hline
		Type & Generators\\ \hline\hline
		$A_{n}$ & $K_{2}^{\varpi_{1}},\hspace{0.35em} K_{3}^{\varpi_{1}},\hspace{0.35em}\ldots,\hspace{0.35em}K_{n+1}^{\varpi_{1}}$\\ \hline
		$B_{n}$ & $K_{2}^{\varpi_{1}},\hspace{0.35em} K_{4}^{\varpi_{1}},\hspace{0.35em}\ldots,\hspace{0.35em}K_{2n}^{\varpi_{1}}$\\ \hline
		$C_{n}$ & $K_{2}^{\varpi_{1}},\hspace{0.35em} K_{4}^{\varpi_{1}},\hspace{0.35em}\ldots,\hspace{0.35em}K_{2n}^{\varpi_{1}}$\\ \hline
		$D_{n}$ & $K_{2}^{\varpi_{1}},\hspace{0.35em} K_{4}^{\varpi_{1}},\hspace{0.35em}\ldots,\hspace{0.35em}K_{2n-2}^{\varpi_{1}}\hspace{0.35em}$ and either $\hspace{0.35em}K_{n}^{\varpi_{n-1}}\hspace{0.35em}$ or $\hspace{0.35em}K_{n}^{\varpi_{n}}\hspace{0.35em}$\\ \hline
		$E_{6}$ & $K_{2}^{\varpi_{1}},\hspace{0.35em} K_{5}^{\varpi_{1}},\hspace{0.35em}K_{6}^{\varpi_{1}},\hspace{0.35em}K_{8}^{\varpi_{1}},\hspace{0.35em}K_{9}^{\varpi_{1}},\hspace{0.35em}K_{12}^{\varpi_{1}}$\\ \hline
		$E_{7}$ & $K_{2}^{\varpi_{1}},\hspace{0.35em} K_{6}^{\varpi_{1}},\hspace{0.35em} K_{8}^{\varpi_{1}},\hspace{0.35em} K_{10}^{\varpi_{1}},\hspace{0.35em} K_{12}^{\varpi_{1}},\hspace{0.35em} K_{14}^{\varpi_{1}},\hspace{0.35em} K_{18}^{\varpi_{1}}$\\ \hline
		$E_{8}$ & $K_{2}^{\varpi_{1}},\, K_{8}^{\varpi_{1}},\hspace{0.35em} K_{12}^{\varpi_{1}},\hspace{0.35em} K_{14}^{\varpi_{1}},\hspace{0.35em} K_{18}^{\varpi_{1}},\hspace{0.35em} K_{20}^{\varpi_{1}},\hspace{0.35em} K_{24}^{\varpi_{1}},\hspace{0.35em} K_{30}^{\varpi_{1}}$\\ \hline
		$F_{4}$ & $K_{2}^{\varpi_{1}},\hspace{0.35em} K_{6}^{\varpi_{1}},\hspace{0.35em} K_{8}^{\varpi_{1}},\hspace{0.35em} K_{12}^{\varpi_{1}}$\\ \hline
		$G_{2}$ & $K_{2}^{\varpi_{1}},\hspace{0.35em} K_{6}^{\varpi_{1}}$\\ \hline
\end{tabular}}
\caption{Generators of $Z(U(\mathfrak{g}))$ for the simple Lie algebras.}\label{tab:ZU_Gens}
\end{table}

For $\mathfrak{g}$ simple, the invariant, symmetric, bilinear trace form $(\cdot,\cdot)_{\kappa}$ must be a scalar multiple of the normalised  Killing form $(\cdot,\cdot)$. We define the (second order) Dynkin index, $I(\kappa)$, to be this ratio  
so that for any $X,Y\in\mathfrak{g}$ we have
\begin{align}\label{DynkinIndex}
	(X,Y)_{\kappa}=I(\kappa)(X,Y).
\end{align}
From this it follows that $\widetilde{X}_{i}^{\kappa}=I(\kappa)^{\shortminus1}\widetilde{X}_{i}$ and thus $K_{2}^{\kappa}=I(\kappa)^{\shortminus1}\sigma_{\kappa}(K_{2})$. In order to determine $I(\kappa)$, we take the trace of $K_{2}^{\kappa}$ over $V(\kappa)$ to obtain
\[
I(\kappa)^{\shortminus1}\hspace{0.1em}\dim (V(\kappa))\,c_{2}(\kappa)
=
I(\kappa)^{\shortminus1}\,\smashoperator{\sum_{1\leq i,j\leq d}}\hspace{0.5em}\left(\widetilde{X}_{i},\widetilde{X}_{j}\right)\left(X_{i}, X_{j}\right)_{\kappa}=\hspace{0.5em}\smashoperator{\sum_{1\leq i,j\leq d}}\hspace{0.5em}\delta_{i,j} = d,
\]
which implies
\begin{align}\label{DynkinIndexValue}
I(\kappa) = \dfrac{\langle\kappa,\kappa+2\rho\rangle \dim V(\kappa)}{d},
\; \text{ or equivalently, }\;
rI(\kappa) = \dfrac{\kappa\cdot(\kappa+2\rho)\dim V(\kappa)}{d}.
\end{align}
In particular, for the adjoint representation with highest root $\theta$, we find $I(\theta)=2h^{\vee}$ for dual Coxeter number $h^{\vee}:= 1+\langle\theta,\rho\rangle=\tfrac{1}{2}\langle\theta,\theta+2\rho\rangle$. Thus we obtain the well-known relation $\Tr_{\mathfrak{g}}(\ad_{X}\circ\ad_{Y})=2h^{\vee}(X,Y)$.

\section{Klimyk's Formula and Okubo's Formula}\label{sec:KlimOk}
We now review formulas due to Klimyk and Okubo that we exploit in our main results. First we present Klimyk's formula \cites{Kl,Sn}, though we omit a proof here. This provides an explicit means of calculating the decomposition of tensor products of irreducible representations. This will be directly applied to Okubo's  formula \cite{Ok} which provides a means of calculating the eigenvalues of generalised Casimir operators using such decompositions. 

\begin{theorem}[Klimyk's formula \cites{Kl,Sn}]\label{theor:Klim}
For $\mathfrak{g}$ a semisimple, complex Lie algebra with dominant, integral elements $\lambda,\kappa,\nu\in\Lambda^{+}$, the multiplicity of the irreducible representation $V(\nu)$ in the decomposition of $V(\lambda)\otimes V(\kappa)$ into irreducible representations is given by
\begin{align}\label{eq:Klim}
	n_{V(\lambda)\hspace{0.08em}\otimes V(\kappa)}(\nu) = \hspace{0.5em} \smashoperator{\sum_{\substack{\mu\in\Lambda(\kappa) \\ \lambda+\mu\sim\nu}}} \hspace{0.5em} m_{\kappa}(\mu)\,\sgn(\lambda+\mu).
\end{align}
\end{theorem}
\begin{remark}
	\begin{enumerate}\leavevmode
		\item[(i)] 
		From \eqref{eq:sgn_hstar} note that $\sgn(\lambda+\mu)\in\{\pm 1\}$ for all of the summands in \eqref{eq:Klim} since $\nu\in\Lambda^{+}$ but that we require knowledge of $W$ to determine these signatures.
		\item[(ii)]  
		Since $\sgn(\lambda+\mu)=0$ for any $\lambda+\mu$ term which does not link to a dominant element, from Remark~\eqref{rem:SignatureRemarks}, the total number of irreducible representations that appear in the decomposition of $V(\lambda)\otimes V(\kappa)$ is given by $\sum_{\mu\in\Lambda(\kappa)}m_{\kappa}(\mu)\sgn(\lambda+\mu)$.
		\item[(iii)]
		An immediate consequence is that $n_{V(\lambda)\hspace{0.08em}\otimes V(\kappa)}(\nu)$ can only be nonzero if $\nu$ lies in the orbit $W\CDOT(\lambda+\mu)$ for some weight $\mu\in\Lambda(\kappa)$ - or by transposing the factors of the tensor representation, lies in the orbit $W\CDOT(\kappa+\mu')$ for some weight $\mu'\in\Lambda(\lambda)$.
		\end{enumerate}
\end{remark}
We next present Okubo's formula \cite{Ok} for the eigenvalue $c_{p}^{\kappa}(\lambda)$ of $K_{p}^{\kappa}$ of \eqref{GeneralisedCasimir} acting on the representation $\Lambda(\lambda)$ for $\lambda,\kappa\in\Lambda^{+}$ for all $p\in \Z_{\ge 0}$. 
The expression relies on the irreducible decomposition of $V(\lambda)\otimes V(\kappa)$ and the eigenvalues of $K_{2}^{\kappa}$. Here, we do include a proof due to significant notational changes from \cite{Ok} and to highlight the connection between the generalised Casimir operator $K_{p}^{\kappa}$ and tensor representations.
\begin{theorem}[Okubo's formula \cite{Ok}]\label{theor:Ok}
Let $\mathfrak{g}$ be a semisimple, complex Lie algebra. %with a basis $\lbrace X_{i}\rbrace_{i=1}^{d}$ and with 
For dominant, integral elements $\lambda,\kappa\in\Lambda^{+}$, the eigenvalue $c^{\kappa}_{p}(\lambda)$ of the action of $K_{p}^{\kappa}$ on the irreducible representation $V(\lambda)$ is given by:
\begin{align*}
	c_{p}^{\kappa}(\lambda) =\sum_{\nu\in\Lambda^{+}}
	\dfrac{n_{V(\lambda)\otimes V(\kappa)}(\nu)\dim V(\nu)}
	{\dim V(\lambda)} \hspace{0.1em}
\,\xi_{\lambda,\kappa}(\nu)^{p},
\end{align*}
where $\xi_{\lambda,\kappa}(\nu):=\tfrac{1}{2}\left(c_{2}^{\kappa}(\nu) -c_{2}^{\kappa}(\lambda)-c_{2}^{\kappa}(\kappa)\right)$. If $\mathfrak{g}$ is simple, we have
\begin{align}\label{XiValueSimple}
	\xi_{\lambda,\kappa}(\nu)=\left(2rI(\kappa)\right)^{\shortminus1}\left((\nu+\rho)^{2}-(\lambda+\rho)^{2}-(\kappa+\rho)^{2}+\rho^{2}\right).
\end{align}
\end{theorem}
\begin{proof}
The result is trivially true for $p=0$ since $	c_{0}^{\kappa}(\lambda)=\dim V(\kappa)$.
Define the following operator in $\End(V(\lambda)\otimes V(\kappa))$:
\begin{align}\label{QOperator}
	Q_{\lambda,\kappa}:&= \sum_{1\leq i\leq d}\sigma_{\lambda}(X_{i})\otimes\sigma_{\kappa}(\widetilde{X}^{\kappa}_{i})\\
	&=\sum_{1\leq i,j\leq d}\left(\widetilde{X}_{i}^{\kappa},\widetilde{X}_{j}^{\kappa}\right)_{\kappa}\sigma_{\lambda}(X_{i})\otimes\sigma_{\kappa}(X_{j}).\nonumber
\end{align}
We have that $Q_{\lambda,\kappa}$ is basis independent and is central in $\End(V(\lambda)\otimes V(\kappa))$ and can be rewritten in terms of degree $2$ Casimir operators:
\begin{align*}
	Q_{\lambda,\kappa}=\tfrac{1}{2}\left(\sigma_{V(\lambda)\hspace{0.08em}\otimes V(\kappa)}(K^{\kappa}_{2})-\sigma_{\lambda}(K^{\kappa}_{2})\otimes \id_{V(\kappa)}-\id_{V(\lambda)}\otimes\sigma_{\kappa}(K^{\kappa}_{2})\right).
\end{align*}
By Schur's Lemma,
$Q_{\lambda,\kappa}$ acts on a component $V(\nu)$ in the irreducible decomposition of $V(\lambda)\otimes V(\kappa)$  as scalar multiplication by $\xi_{\lambda,\kappa}(\nu)$. This implies that for $p\in \Z^{+}$ we have
\begin{align}\label{QpExpression1}
%	\Tr_{\hspace{0.1em}V(\lambda)\otimes V(\kappa)\hspace{0.1em}}(Q_{\lambda,\kappa}) &= \sum_{\nu\in\Lambda^{+}}n_{V(\lambda)\otimes V(\kappa)}(\nu)\hspace{0.1em}\dim V(\nu)\,\xi_{\lambda,\kappa}(\nu),\nonumber\\
%	&\hspace{2em}\text{and}\nonumber\\
	\Tr_{\hspace{0.1em}V(\lambda)\hspace{0.08em}\otimes V(\kappa)\hspace{0.1em}}(Q_{\lambda,\kappa}^{p}) &= \sum_{\nu\in\Lambda^{+}}n_{V(\lambda)\hspace{0.08em}\otimes V(\kappa)}(\nu)\hspace{0.1em}\dim V(\nu)\,\xi_{\lambda,\kappa}(\nu)^{p}.
\end{align}
The $p^{\text{th}}$ power of $Q_{\lambda,\kappa}$ can be written as the sum
\begin{align}\label{QPower}
	Q_{\lambda,\kappa}^{p}= \hspace{1em}\smashoperator{\sum_{1\leq i_{1},\ldots,i_{p}\leq d}}\hspace{1em}\sigma_{\lambda}(X_{i_{1}}\otimes\cdots\otimes X_{i_{p}})\otimes\sigma_{\kappa}(\widetilde{X}_{i_{1}}^{\kappa}\otimes\cdots\otimes\widetilde{X}_{i_{p}}^{\kappa}).
\end{align}
Computing the partial trace over $V(\kappa)$ gives the image in $\End V(\lambda)$ of the generalised, degree $p\ge 1$ Casimir operator $K_{p}^{\kappa}$ under the representation map $\sigma_{\lambda}$ as
\begin{align*}
	\Tr_{V(\kappa)}\left(Q_{\lambda,\kappa}^{p}\right)
	\hspace{0.5em}= \hspace{1em}\smashoperator{\sum_{1\leq i_{1},\ldots,i_{p}\leq d}}\hspace{1em}\left(\widetilde{X}_{i_{1}}^{\kappa},\cdots,\widetilde{X}_{i_{p}}^{\kappa}\right)_{\kappa}\sigma_{\lambda}(X_{i_{1}}\otimes\cdots\otimes X_{i_{p}})\hspace{0.5em}=\hspace{0.5em}\sigma_{\lambda}(K_{p}^{\kappa}).
\end{align*}
It follows, in contrast to \eqref{QpExpression1}, that the full trace of $Q_{\lambda,\kappa}^{p}$ can be expressed as
\begin{align}\label{QpExpression2}
	\Tr_{\hspace{0.1em}V(\lambda)\hspace{0.08em}\otimes V(\kappa)\hspace{0.1em}}(Q_{\lambda,\kappa}^{p}) \hspace{0.5em}=\hspace{0.5em} \Tr_{\hspace{0.1em}V(\lambda)\hspace{0.1em}}(\sigma_{\lambda}(K_{p}^{\kappa})) \hspace{0.5em}=\hspace{0.45em} \dim V(\lambda)\,c^{\kappa}_{p}(\lambda).
\end{align}
The main result follows from equating the expressions in \eqref{QpExpression1} and \eqref{QpExpression2}.

If $\mathfrak{g}$ is simple, then %recall the Dynkin index \eqref{DynkinIndex} and 
\eqref{DynkinIndexValue} gives a relation between the eigenvalues of generalised quadratic Casimir operators and the standard quadratic Casimir operator via  $c_{2}^{\kappa}(\nu)=I(\kappa)^{\shortminus1}c_{2}(\nu)$ for $\nu\in\Lambda^{+}$. Thus
%\begin{align*}
%	c_{2}^{\kappa}(\nu) &= I(\kappa)^{\shortminus1}\langle\nu,\nu+2\rho\rangle= \left(rI(\kappa)\right)^{\shortminus1}\nu\cdot(\nu+2\rho)= \left(rI(\kappa)\right)^{\shortminus1}\left((\nu+\rho)^{2}-\rho^{2}\right)\nonumber\\[6pt]
%	&\Rightarrow \hspace{0.5em}\xi_{\lambda,\kappa}(\nu) = \left(2rI(\kappa)\right)^{\shortminus1} \left((\nu+\rho)^{2}-(\lambda+\rho)^{2}-(\kappa+\rho)^{2}+\rho^{2}\right),
%\end{align*}
\begin{align*}
c_{2}^{\kappa}(\nu) &= I(\kappa)^{\shortminus1}\langle\nu,\nu+2\rho\rangle= \left(rI(\kappa)\right)^{\shortminus1}\nu\cdot(\nu+2\rho)= \left(rI(\kappa)\right)^{\shortminus1}\left((\nu+\rho)^{2}-\rho^{2}\right)\nonumber,
\end{align*}
where $r$ is given by $\mu\cdot\mu'=r\langle\mu,\mu'\rangle$ for $\mu,\mu'\in\mathfrak{h}^{*}$. Hence \eqref{XiValueSimple} follows.
\end{proof}
The $Q_{\lambda,\kappa}$ operator that appears in this proof of Okubo's formula also provides a way to prove a relationship between the eigenvalues $c_{p}^{\kappa}(\lambda)$ and $c_{p}^{\lambda}(\kappa)$ referred to as the ``Reciprocity theorem'' by Perelomov and Popov \cite{PP}.
\begin{lemma}\label{lem:recip}
	For $\mathfrak{g}$ simple and any $\lambda,\kappa\in\Lambda^{+}$, we have that:
	\begin{align*}
		c_{p}^{\kappa}(\lambda) = \left(\dfrac{\dim V(\lambda)}{\dim V(\kappa)}\right)^{p\shortminus1}\left(\dfrac{\lambda\cdot(\lambda+2\rho)}{\kappa\cdot(\kappa+2\rho)}\right)^{p}c_{p}^{\lambda}(\kappa).
	\end{align*}
\end{lemma}
\begin{proof} The result is trivially true for $p=0$. %Let $Q_{\lambda,\kappa}$ be as in \eqref{QOperator}. 
	By \eqref{QPower} we have for $p\in \Z^{+}$ that
	\begin{align*}
		I(\kappa)^{p}\,Q_{\lambda,\kappa}^{p} = \sum_{1\leq i_{1},\ldots,i_{p}\leq d} \sigma_{\lambda}(X_{i_{1}}\otimes\cdots\otimes X_{i_{p}})\otimes \sigma_{\kappa}(\widetilde{X}_{i_{1}}\otimes\cdots\otimes \widetilde{X}_{i_{p}}).
	\end{align*}
	Since $Q_{\lambda,\kappa}$ is basis independent we can swap $\lambda$ and $\kappa$ and find
	\begin{align*}
		I(\lambda)^{p}\,Q_{\kappa,\lambda}^{p}= \sum_{1\leq i_{1},\ldots,i_{p}\leq d}  \sigma_{\kappa}(\widetilde{X}_{i_{1}}\otimes\cdots\otimes \widetilde{X}_{i_{p}})\otimes\sigma_{\lambda}(X_{i_{1}}\otimes\cdots\otimes X_{i_{p}}).
	\end{align*}
Since $V(\lambda)\otimes V(\kappa)\cong V(\kappa)\otimes V(\lambda)$ as representations, we find
	\begin{align*}
		I(\kappa)^{p}\,\Tr_{V(\lambda)\otimes V(\kappa)}\left(Q_{\lambda,\kappa}^{p}\right) = I(\lambda)^{p}\,\Tr_{\hspace{0.1em}V(\kappa)\otimes V(\lambda)\hspace{0.1em}}\left(Q_{\kappa,\lambda}^{p}\right).
	\end{align*}
	So by \eqref{QpExpression2} we have
	\begin{align*}
		I(\kappa)^{p}\hspace{0.1em}\dim V(\lambda)\,c_{p}^{\kappa}(\lambda) \hspace{0.15em}=\hspace{0.15em} I(\lambda)^{p} \hspace{0.1em}\dim V(\kappa)\,c_{p}^{\lambda}(\kappa).
	\end{align*}
	The result then follows from \eqref{DynkinIndexValue}.
\end{proof}

\section{Generating Functions of Casimir Invariants}\label{sec:GenFun}
This section contains our main result, Theorem~\ref{theor:CasGen}, giving a novel expression for the generating function of the generalised Casimir operator eigenvalues $c_{p}^{\kappa}(\lambda)$  for $\lambda,\kappa\in\Lambda^{+}$. 
This expression exploits natural extensions of $\dim V(\nu)$ of \eqref{DimensionFormula} and $\xi_{\lambda,\kappa}(\nu)$ of  \eqref{XiValueSimple} as functions of $\nu\in\Lambda^{+}$ to functions on  $\mathfrak{h}^{*}$. Dealing first with the former, for $\mu \in \mathfrak{h}^{*}$ we define

\begin{align}\label{eq:D_defn}
	D(\mu):=\prod_{\alpha\in\Delta^{+}}\dfrac{\langle\mu+\rho,\alpha\rangle}{\langle\rho,\alpha\rangle}.
\end{align}

Clearly $D(\lambda)=\dim V(\lambda)$ for $\lambda\in\Lambda^{+}$ by the Weyl dimension formula \eqref{DimensionFormula}. Recalling the signature \eqref{eq:sgn_hstar}, we consider elements $\mu\in\mathfrak{h}^{*}$ with $\sgn(\mu)=0$, then with $\sgn(\mu)\neq0$ each in turn. Recall, from Remark~\eqref{rem:SignatureRemarks}, that $\sgn(\mu)=0$ iff $\mu$ lies in the boundary of Weyl chambers and $\sgn(\mu)\neq0$ iff $\mu\sim\nu$ for some $\nu\in\mathfrak{h}^{*}$ such that $\nu+\rho$ is strongly dominant.
\begin{lemma}\label{lem:ZeroSignature}
	For $\mu\in\mathfrak{h}^{*}$, then $\sgn(\mu)=0$ if and only if $D(\mu)=0$.
\end{lemma}
\begin{proof}
	We have $\sgn(\mu)=0$ iff $\mu+\rho$ lies outside of the Weyl chambers, with the complement of the Weyl chambers being the union of all hyperplanes in $\mathfrak{h}^{*}$ perpendicular to a positive root $\alpha\in\Delta^{+}$. That is, $\sgn(\mu)=0$ iff there is some $\alpha\in\Delta^{+}$ such that $\langle\mu+\rho,\alpha\rangle=0$. Then by definition \eqref{eq:D_defn}, $\sgn(\mu)=0$ iff $D(\mu)=0$.
\end{proof}

\begin{lemma}\label{lem:NonZeroSignature}
If $\mu\sim\nu$ for $\mu,\nu\in\mathfrak{h}^{*}$ such that $\nu+\rho\in\FWC$ then $D(\mu)=	\sgn(\mu)D(\nu)$.
\end{lemma}	 
\begin{proof}
Since $\nu+\rho\in\FWC$, there is a unique linkage $w$ such that $\mu+\rho=w(\nu+\rho)$ and, by \eqref{eq:sgn_hstar}, we have $\sgn(\mu)=\sgn(w)$. By Lemma B in Section $10.2$ of Humphreys \cite{Hu1} for $1\leq i\leq n$, we know for the standard Weyl reflection  $s_{i}(\alpha_{i})=-\alpha_{i}$ that $s_{i}$ permutes $\Delta^{+}\setminus\lbrace\alpha_{i}\rbrace$ i.e. the other positive roots. Since, for any $\tau,\tau'\in\mathfrak{h}^{*}$, we have $\langle \tau,\tau'\rangle=\langle s_{i}(\tau),s_{i}(\tau')\rangle$,  we find
\begin{align*}
	\prod_{\alpha\in\Delta^{+}}\langle \tau,\alpha\rangle 
	= \prod_{\alpha\in\Delta^{+}}\langle s_{i}(\tau),s_{i}(\alpha)\rangle 
	= -\hspace{0.125em}\smashoperator[l]{\prod_{\alpha\in\Delta^{+}}}\langle s_{i}(\tau),\alpha\rangle =\hspace{0.125em}\sgn(s_{i})\hspace{0.125em}\smashoperator[l]{\prod_{\alpha\in\Delta^{+}}}\langle s_{i}(\tau),\alpha\rangle.
\end{align*} 
Since $W$ is generated by the simple Weyl reflections and the signature is multiplicative,
\begin{align*}
	\prod_{\alpha\in\Delta^{+}}\langle \tau,\alpha\rangle =\sgn(w)\prod_{\alpha\in\Delta^{+}}\langle w(\tau),\alpha\rangle,
	\mbox{ for all } w\in W.
\end{align*}
Taking $\tau=\mu+\rho$ and using the definition \eqref{eq:D_defn} of $D(\mu)$, the result follows.
\end{proof}

Note that $D(\mu)=\sgn(w)D(\nu)$ holds for any $\mu\sim\nu$ with linkage $w$, though it may not be the case that $\sgn(\mu)=\sgn(w)$, for example if $\sgn(\mu)=0$ then $D(\mu)=D(\nu)=0$ by Lemma~\eqref{lem:ZeroSignature}. Recall from Remark~\eqref{rem:SignatureRemarks} that $\mu\in\Lambda$ has $\sgn(\mu)\neq0$ iff $\mu$ is linked to a dominant integral element $\lambda\in\Lambda^{+}$. Then Lemmas~\ref{lem:ZeroSignature} and ~\ref{lem:NonZeroSignature} give us

\begin{corollary}\label{cor:DLambda}
	Let $\mu\in\Lambda$. If $\mu\sim\lambda$ for some $\lambda\in\Lambda^{+}$ then $D(\mu)=\sgn(\mu)\dim V(\lambda)$. Otherwise $D(\mu)=\sgn(\mu)=0$.
\end{corollary}
Corollary~\ref{cor:DLambda} allows us determine $\sgn(\mu)$ for any $\mu\in\Lambda$, and hence, through Klimyk's formula  (Theorem~\ref{theor:Klim}), to compute $\dim V(\lambda)\otimes V(\kappa)$ for $\lambda,\kappa\in\Lambda^{+}$ as:
\begin{align*}
	\dim V(\lambda)\dim V(\kappa)&= \smashoperator[r]{\sum_{\nu\in\Lambda^{+}}}\,n_{V(\lambda)\hspace{0.08em}\otimes V(\kappa)}(\nu)\hspace{0.1em}\dim V(\nu)\\[3pt]
	&= \smashoperator[r]{\sum_{\nu\in\Lambda^{+}}}\hspace{0.35em}\smashoperator[r]{\sum_{\substack{\mu\in\Lambda(\kappa) \\ \lambda+\mu\sim\nu}}}\, m_{\kappa}(\mu)\,\sgn(\lambda+\mu)\hspace{0.1em}\dim V(\nu)\\[3pt]
	&= \smashoperator[r]{\sum_{\nu\in\Lambda^{+}}}\hspace{0.35em}\smashoperator[r]{\sum_{\substack{\mu\in\Lambda(\kappa) \\ \lambda+\mu\sim\nu}}}\, m_{\kappa}(\mu)\,D(\lambda+\mu),
%	&= \smashoperator[r]{\sum_{\substack{\mu\in\Lambda(\kappa) \\ \lambda+\mu\in W\CDOT\Lambda^{+}}}}\hspace{0.75em}m_{\kappa}(\mu)\,\sgn(\lambda+\mu)\hspace{0.1em}\dim V(w_{\lambda+\mu}\CDOT\left(\lambda+\mu\right)),
\end{align*}
Since $D(\lambda+\mu)=0$ if $\lambda+\mu$ is not linked to any $\nu\in\Lambda^{+}$, we may simplify this further as
%
%Since, whenever $\lambda+\mu\notin W\CDOT\Lambda^{+}$ we have $\sgn(\lambda+\mu)=0$, we can simplify this expression further by formally defining such a $w_{\lambda+\mu}$ to map to $0$. Then
%\begin{align}\label{TensorDimension}
%	\smashoperator[r]{\sum_{\nu\in\Lambda^{+}}}n_{V(\lambda)\otimes V(\kappa)}(\nu)\dim V(\nu) =\smashoperator[lr]{\sum_{\mu\in\Lambda(\kappa)}}m_{\kappa}(\mu)\hspace{0.1em}\sgn(\lambda+\mu)\dim V (w_{\lambda+\mu}\CDOT(\lambda+\mu) ).
%\end{align}
%This approach motivates examining how the dimension of $V(w_{\tau}\CDOT\tau)$ could be directly expressed in terms of $\tau$ for arbitrary $\tau\in\Lambda$. Currently, we must both determine if a linkage to a dominant element exists and what the linkage is, when only the signature $\sgn(\tau)$ and the dimension  $\dim V(w_{\tau}\CDOT\tau) $ are necessary.\\
%****** \\This  is a little unclear IMHO. Suggest we just rewrite Klimyk as you have later on in \eqref{NewTensorDimension} once zero signatures are introduced
%
%
%
%
%
%
%This addresses the points raised after our statement of Klimyk's formula (Theorem~\ref{theor:Klim}). Indeed, \eqref{TensorDimension} now simplifies to
\begin{align}\label{NewTensorDimension}
	\dim V(\lambda)\dim V(\kappa) = \sum_{\mu\in\Lambda(\kappa)}m_{\kappa}(\mu)\,D(\lambda+\mu).
\end{align}
This expression does not require any knowledge of the Weyl group. 
Furthermore, when computing the right hand sum of \eqref{NewTensorDimension}, if we first partition the summands based on the absolute value of $D(\lambda+\mu)$ and compute the sum over each part separately, we will have determined the number of components of a given dimension that appear in the decomposition of $V(\lambda)\otimes V(\kappa)$. While not enough to completely determine the decomposition, especially if $\mathfrak{g}$ has non-trivial outer automorphisms, it does greatly simplify the process.

We know from Klimyk's formula that the term in Okubo's formula associated with $\nu$ is non-zero only when $\nu\in\Lambda^{+}$ and is linked to $\lambda+\mu$ for some weight $\mu\in\Lambda(\kappa)$. That is, all non-zero $\xi_{\lambda,\kappa}(\nu)$ terms  are of the form $\xi_{\lambda,\kappa}\left(w \CDOT(\lambda+\mu)\right)$ for $w$ the linkage between $\lambda+\mu$ and a dominant integral element.
For $\mathfrak{g}$ simple, to study these expressions we extend $\xi_{\lambda,\kappa}:\Lambda^{+}\rightarrow\R$ to all of $\mathfrak{h}^{*}$ using \eqref{XiValueSimple}. Thus for $\tau\in\mathfrak{h}^{*}$ and $\lambda,\kappa\in\Lambda^{+}$ we define
\begin{align}\label{eq:Xi_defn}
	\xi_{\lambda,\kappa}(\tau) := \left(2rI(\kappa)\right)^{\shortminus1} \left((\tau+\rho)^{2}-(\lambda+\rho)^{2}-(\kappa+\rho)^{2}+\rho^{2}\right).
\end{align}

\begin{lemma}\label{lem:xi_invariance}
For $\mathfrak{g}$ simple, $\xi_{\lambda,\kappa}$ is invariant under the dot action of $W$.
\end{lemma}
\begin{proof}
As the standard Weyl action is length preserving, for any $w\in W$ and $\tau\in\mathfrak{h}^{*}$, we have $w(\tau+\rho)^{2}=r\Vert w(\tau+\rho)\Vert^{2}=r\Vert\tau+\rho\Vert^{2}=(\tau+\rho)^{2}$. So
\begin{align*}
	\xi_{\lambda,\kappa}(w\CDOT\tau) &= (2rI(\kappa))^{\shortminus1}\left(w(\tau+\rho)^{2}-(\lambda+\rho)^{2}-(\kappa+\rho)^{2}+\rho^{2}\right)\\
	&=(2rI(\kappa))^{\shortminus1}\left((\tau+\rho)^{2}-(\lambda+\rho)^{2}-(\kappa+\rho)^{2}+\rho^{2}\right) \hspace{0.5em} = \xi_{\lambda,\kappa}(\tau).
\end{align*}
\end{proof}
Note Lemma~\ref{lem:xi_invariance} implies $\xi_{\lambda,\kappa}$ gives an invariant $S(\mathfrak{h})^{W}\rightarrow\mathbb{C}$ on the $W$-invariant subalgebra of the symmetric algebra on $\mathfrak{h}$. The Harish-Chandra isomorphism, $\phi: Z(U(\mathfrak{g}))\rightarrow S(\mathfrak{h})^{W}$, identifies the centre of $U(\mathfrak{g})$ with this $W$-invariant subalgebra via a twisted vector space projection \cite{Kn}. Thus any central character, that is the induced map $Z(U(\mathfrak{g}))\rightarrow\mathbb{C}$ for any representation of $\mathfrak{g}$ for which central elements of $U(\mathfrak{g})$ act by scalar multiplication - in particular irreducible representations by Schur's Lemma - can be identified with an invariant $S(\mathfrak{h})^{W}\rightarrow\mathbb{C}$, such as that defined by $\xi_{\lambda,\kappa}$.

The preceding results give us a way to calculate the Casimir eigenvalues  $c_{p}^{\kappa}(\lambda)$, when $\mathfrak{g}$ is simple, depending only on knowledge of $\Delta^{+}$ (in determining the $D(\nu)$ terms) and of the weights $\Lambda(\kappa)$, counting multiplicities. We do not require any direct knowledge of, or computation using, the Weyl group and thus bypass a significant computational barrier. We state this, for a given associated irreducible representation of highest weight $\kappa$ and an irreducible representation of highest weight $\lambda$ being acted upon by the generalised Casimir operator $K_{p}^{\kappa}$, in terms of an ordinary generating function defined by
%over all degrees as Okubo's formula readily permits. That is, for a given simple Lie algebra $\mathfrak{g}$ and $\lambda,\kappa\in\Lambda^{+}$ we define the formal sum:
\begin{align}\label{GenFuncDef}
	G_{\mathfrak{g}}^{\kappa}(\lambda,z):=\sum_{p\geq 0}c_{p}^{\kappa}(\lambda)z^{p}.
\end{align}
Let $\Chat:=\C \cup \{\infty\}$ denote the extended complex plane.
\begin{theorem}[The Casimir Eigenvalue Generating Function]\label{theor:CasGen}
For $\mathfrak{g}$ a finite-dimensional, simple, complex Lie algebra and dominant, integral elements $\lambda,\kappa\in\Lambda^{+}$, the generating function for the eigenvalues of the action of the generalised Casimirs associated with $V(\kappa)$ acting on $V(\lambda)$ is given by the following rational function of $z\in\Chat$ 
\begin{align}\label{eq:Gnew}
	G_{\mathfrak{g}}^{\kappa}(\lambda,z) =
	\dfrac{1}{\dim V(\lambda)}
	\sum_{\mu\in\Lambda(\kappa)}\dfrac{m_{\kappa}(\mu)\,D(\lambda+\mu)}{1-z\,\xi_{\lambda,\kappa}(\lambda+\mu) },
\end{align}	
for $D$ of \eqref{eq:D_defn} and  $ \xi_{\lambda,\kappa}$ of \eqref{eq:Xi_defn}.
%=\hspace{0.1em} \left(2rI(\kappa)\right)^{\shortminus1}\left((\tau+\rho)^{2}-(\lambda+\rho)^{2}-(\kappa+\rho)^{2}+\rho^{2}\right)\hspace{0.35em}$ for any choice of $\tau\in\Lambda$ and the Dynkin index is given by $I(\kappa) = D(\kappa)\langle\kappa,\kappa+2\rho\rangle/d$.
\end{theorem}
\begin{proof}
From Klimyk's and Okubo's formulas (Theorems~\ref{theor:Klim} and \ref{theor:Ok}) we find
\begin{align*}
	c_{p}^{\kappa}(\lambda)&=
	\sum_{\nu\in\Lambda^{+}}\hspace{0.75em} \smashoperator{\sum_{\substack{\mu\in\Lambda(\kappa) \\ \lambda+\mu\sim\nu}}}\hspace{0.5em} 
	\dfrac{m_{\kappa}(\mu)\, \sgn(\lambda+\mu)\,D(\nu)}{\dim V(\lambda)}\; \left(\xi_{\lambda,\kappa}(\nu)\right)^{p}.
\end{align*}
By applying Corollary~\ref{cor:DLambda} and Lemma~\ref{lem:xi_invariance}, this expression then becomes
\begin{align*}
	c_{p}^{\kappa}(\lambda)&=
	\hspace{0.75em} \smashoperator{\sum_{\substack{\mu\in\Lambda(\kappa) \\ \lambda+\mu\in W\CDOT\Lambda^{+}}}}\hspace{0.5em} 
	\dfrac{m_{\kappa}(\mu)\, D(\lambda+\mu)}{\dim V(\lambda)}\; \left(\xi_{\lambda,\kappa}(\lambda+\mu)\right)^{p}.
\end{align*}
Since $D(\lambda+\mu)=0$ when $\lambda+\mu\notin W\CDOT\Lambda^{+}$, by Corollary~\ref{cor:DLambda}, we have
\begin{align*}
	c_{p}^{\kappa}(\lambda)&=
%	\dfrac{1}	{\dim V(\lambda)}
	\hspace{0.75em} \smashoperator{\sum_{\mu\in\Lambda(\kappa)}}\hspace{0.5em} \dfrac{m_{\kappa}(\mu)\, D(\lambda+\mu)}{\dim V(\lambda)}\; \left(\xi_{\lambda,\kappa}(\lambda+\mu)\right)^{p}.
\end{align*}
Thus $G_{\mathfrak{g}}^{\kappa}(\lambda,z)$ of \eqref{GenFuncDef} is the rational function of $z\in \Chat$ given by \eqref{eq:Gnew}.
\end{proof}
The summation in \eqref{eq:Gnew} is computationally manageable provided that $\Lambda(\kappa)$ is not too large. Note that, if all that is desired is the generating function for some specific $\lambda$ and $\kappa$, it may be more efficient, when $\vert\Lambda(\lambda)\vert<\vert\Lambda(\kappa)\vert$, to swap their roles and apply the reciprocity formula of Lemma~\ref{lem:recip}. In fact, the definition \eqref{GenFuncDef} and  Lemma~\ref{lem:recip} imply the following generating function relation:
\begin{corollary}
For $\mathfrak{g}$ simple, we have $\,G_{\mathfrak{g}}^{\kappa}(\lambda,z) \,=\, 
\tfrac{\dim V(\kappa)}	{\dim V(\lambda)}
\,G_{\mathfrak{g}}^{\lambda}\left(\kappa,\tfrac{I(\lambda)}{I(\kappa)}\hspace{0.125em}z\right)$. 
%$\dim V(\lambda)\,G_{\mathfrak{g}}^{\kappa}(\lambda,I(\kappa)z) \,=\, \dim V(\kappa)\,G_{\mathfrak{g}}^{\lambda}(\kappa,I(\lambda)z)$.
\end{corollary}

From Theorem~\ref{theor:CasGen} we find that the rational function $G_{\mathfrak{g}}^{\kappa}(\lambda,z)$ has a simple pole expansion in $z$ and converges at the point at infinity.
We use this fact  to directly derive previously known closed form generating functions due to Perelomov and Popov \cite{PP} in the case of the classical Lie algebras for generalised Casimir operators associated with the first fundamental representation $\kappa=\varpi_{1}$. To demonstrate this, and in the interest of examining particular generating functions, we define the following useful pieces of notation. For  $\lambda,\kappa\in\Lambda^{+}$, $\mu\in\Lambda(\kappa)$ and $z\in\Chat$, we define 
\begin{align}
	\ell&:=\lambda+\rho,\label{eq:ell}\\[6pt]
	\psi^{\kappa}(\mu)&:=\kappa\cdot\rho+\tfrac{1}{2}(\kappa^{2}-\mu^{2}),\label{eq:psi}\\[6pt]
	\zeta&:=\dfrac{rI(\kappa)}{z}+\kappa\cdot\rho.\label{eq:zeta}
\end{align}
We then find that
\begin{align*}
	\xi_{\lambda,\kappa}(\lambda+\mu) 
&= \left(2rI(\kappa)\right)^{\shortminus1}\left((\lambda+\mu+\rho)^{2}-(\lambda+\rho)^{2}-(\kappa+\rho)^{2}+\rho^{2}\right)\\
%	&=\left(rI(\kappa)\right)^{\shortminus1}\left((\lambda+\rho)\cdot\mu+\tfrac{1}{2}\mu^{2}-\tfrac{1}{2}\kappa^{2}-\kappa\cdot\rho\right)\\
	&=\left(rI(\kappa)\right)^{\shortminus1}\left(\ell\cdot\mu-\psi^{\kappa}(\mu)\right).
\end{align*}
If $\Lambda(\kappa)=W(\kappa)$, such as  for $\kappa=\varpi_{1}$ for Lie algebras of type $A_{n}$, $C_{n}$, $D_{n}$, $E_{6}$ or $E_{7}$, then we have $\mu^{2}=\kappa^{2}$ and $m_{\kappa}(\mu)=1$ for all $\mu\in\Lambda(\kappa)$. So $\psi^{\kappa}(\mu)=\kappa\cdot\rho$ and thus
\begin{align*}
	\dfrac{z}{rI(\kappa)}G_{\mathfrak{g}}^{\kappa}(\lambda,z) &= \sum_{\mu\in\Lambda(\kappa)}\dfrac{z\,m_{\kappa}(\mu)}{rI(\kappa)-\left(\ell\cdot\mu-\psi^{\kappa}(\mu)\right)z}\,\dfrac{D(\lambda+\mu)}{D(\lambda)}\\[4pt]
	&=\sum_{\mu\in W(\kappa)}\dfrac{1}{\zeta-\ell\cdot\mu}\hspace{0.75em}\smashoperator{\prod_{\alpha\in\Delta^{+}}}\hspace{0.125em}\dfrac{(\ell+\mu)\cdot\alpha}{\ell\cdot\alpha}.
\end{align*}
If $\Lambda(\kappa)=W(\kappa)\cup\lbrace0\rbrace$, such as for $\kappa=\varpi_{1}$ for Lie algebras of type $B_{n}$, $E_{8}$, $F_{4}$ or $G_{2}$, there is only a single additional $\mu=0$ weight to consider beyond the Weyl orbit of $\kappa$.  The summand associated with $\mu=0$ in $\tfrac{z}{rI(\kappa)}G_{\mathfrak{g}}^{\kappa}(\lambda,z)$ is
\begin{align*}
	\dfrac{z\,m_{\kappa}(0)}{rI(\kappa)+\left(\kappa\cdot\rho+\tfrac{1}{2}\kappa^{2}\right)z}\hspace{0.5em}\dfrac{D(\lambda)}{D(\lambda)}\, = \,\dfrac{m_{\kappa}(0)}{\zeta+\tfrac{1}{2}\kappa^{2}}.
\end{align*}
The last couple of expressions can be summarised by the following:
\begin{corollary}\label{cor:GenWkappa}
For $\mathfrak{g}$ simple with $\Lambda(\kappa)\subset W(\kappa)\cup\lbrace0\rbrace$, we have that
\begin{align*}
		\dfrac{z}{rI(\kappa)}G_{\mathfrak{g}}^{\kappa}(\lambda,z) =
	  \dfrac{m_{\kappa}(0)}{\zeta+\tfrac{1}{2}\kappa^{2}}\, + \sum_{\mu\in W(\kappa)}\dfrac{1}{\zeta-\ell\cdot\mu}\hspace{0.75em}\smashoperator{\prod_{\alpha\in\Delta^{+}}}\hspace{0.125em}\dfrac{(\ell+\mu)\cdot\alpha}{\ell\cdot\alpha}.
\end{align*}
\end{corollary}
We will shortly apply Corollary~\ref{cor:GenWkappa} to the classical Lie algebras to derive closed form expressions for the generating functions of generalised Casimir operators associated with the first  fundamental representation $V(\varpi_{1})$. These generating functions were originally due to Perelomov and Popov \cite{PP} but were computed from detailed triangularisation of representation matrices whereas Corollary~\ref{cor:GenWkappa} gives a more direct approach. 

Consider $\kappa=\varpi_{1}$, the highest weight  for the first fundamental representation. Since  $G_{\mathfrak{g}}^{\varpi_{1}}(\lambda,z)$ is a power series in $z$, we note
\begin{align*}
	\lim_{z\rightarrow0}\tfrac{z}{rI(\varpi_{1})}G_{\mathfrak{g}}^{\varpi_{1}}(\lambda,z)=\lim_{\vert\zeta\vert\rightarrow\infty}\tfrac{z}{rI(\varpi_{1})}G_{\mathfrak{g}}^{\varpi_{1}}(\lambda,z)=0.
\end{align*}
Suppose that there exists $f(\zeta)$ with the same simple poles and residues as  $\tfrac{z}{rI(\varpi_{1})}G_{\mathfrak{g}}^{\varpi_{1}}(\lambda,z)$  and which satisfies $\lim_{\vert\zeta\vert\rightarrow\infty}f(\zeta)=0$. Then we can conclude that
\begin{align*}
	G_{\mathfrak{g}}^{\varpi_{1}}(\lambda,z) = \left(\zeta-\varpi_{1}\cdot\rho\right)f(\zeta) = \dfrac{rI(\varpi_{1})}{z}f\left(\dfrac{rI(\varpi_{1})}{z}+\varpi_{1}\cdot\rho\right).
\end{align*}
\begin{remark}\label{rem:GenFuncProperties}
If there exists $\mu\in\Lambda(\varpi_{1})$ for which $D(\lambda+\mu)\neq 0$ and $\xi_{\lambda,\kappa}(\lambda+\mu)=0$ (or equivalently, $\psi^{\varpi_{1}}(\mu)=\ell\cdot\mu$), then $G_{\mathfrak{g}}^{\varpi_{1}}(\lambda,z)$ is non-zero at the point at infinity implying that $f(\zeta)$ has a simple pole at $\zeta=\varpi_{1}\cdot\rho$.
\end{remark}
The pursuit of such an $f(\zeta)$ is the approach we now take in order to find closed expressions for classical $\mathfrak{g}$ which reproduce results of Perelomov and Popov \cite{PP}. 
Tables~\ref{tab:g_data} and \ref{tab:pomega1_data} below provide the relevant data to apply Corollary~\ref{cor:GenWkappa} to any simple Lie algebra for $\kappa=\varpi_{1}$. For the  Euclidean embeddings of Table~\ref{tab:EucEmbed}, we find $\rho_{i+1}=\rho_{i}-1$ for the classical Lie algebras. We also recall the Dynkin index given by $rI(\varpi_{1})=\tfrac{1}{d}\kappa\cdot(\varpi_{1}+2\rho)\dim V(\varpi_{1})$ from \eqref{DynkinIndexValue}. 

%\vspace{1ex}

\begin{table}[ht!]
	{\setlength{\extrarowheight}{3pt}\begin{tabular}{|@{\hspace{0.45em}}c@{\hspace{0.45em}}|@{\hspace{0.9em}}c@{\hspace{0.9em}}|@{\hspace{0.95em}}c@{\hspace{0.95em}}|@{\hspace{1.05em}}c@{\hspace{1.05em}}|@{\hspace{1.25em}}c@{\hspace{1.25em}}|@{\hspace{1.25em}}c@{\hspace{1.25em}}|}\hline
			Type & $\mathfrak{g}$ & $k$ & $\rho$ & $d$ & $r$ \\ \hline\hline
			$A_{n}$ & $\mathfrak{sl}_{n+1}$ & $n+1$ & $\tfrac{n}{2}e_{1}+\tfrac{(n-2)}{2}e_{2}+\cdots-\tfrac{n}{2}e_{n+1}$ & $n(n+2)$ & $1$ \\ \hline
			$B_{n}$ & $\mathfrak{so}_{2n+1}$ & $n$ & $\left(\tfrac{2n-1}{2}\right)e_{1}+\left(\tfrac{2n-3}{2}\right)e_{2}+\cdots+\tfrac{1}{2}e_{n}$ & $n(2n+1)$ & $1$ \\ \hline
			$C_{n}$ & $\mathfrak{sp}_{2n}$ & $n$ & $ne_{1}+(n-1)e_{2}+\cdots+e_{n}$ & $n(2n+1)$ & $2$ \\ \hline
			$D_{n}$ & $\mathfrak{so}_{2n}$ & $n$ & $(n-1)e_{1}+(n-2)e_{2}+\cdots+e_{n-1}$ & $n(2n-1)$ & $1$ \\ \hline
			$E_{6}$ & $\mathfrak{e_{6}}$ & $8$ & $\left(4,3,2,1,0,-4,-4,-4\right)$ & $78$ & $1$ \\ \hline
			$E_{7}$ & $\mathfrak{e_{7}}$ & $8$ & $\left(5,4,3,2,1,0,-\tfrac{17}{2},-\tfrac{17}{2}\right)$ & $133$ & $1$ \\ \hline
			$E_{8}$ & $\mathfrak{e_{8}}$ & $8$ & $\left(6,5,4,3,2,1,0,-23\right)$ & $248$ & $1$ \\ \hline
			$F_{4}$ & $\mathfrak{f_{4}}$ & $4$ & $\left(3,2,1,-8\right)$ & $52$ & $2$ \\ \hline
			$G_{2}$ & $\mathfrak{g_{2}}$ & $3$ & $\left(2,1,-3\right)$ & $14$ & $3$ \\ \hline
	\end{tabular}}
	\caption{Summary of data on simple Lie algebras.}\label{tab:g_data}
\end{table}

%\vspace{1ex}

\begin{table}[ht!]
	{\setlength{\extrarowheight}{3pt}\begin{tabular}{|@{\hspace{0.45em}}c@{\hspace{0.45em}}|@{\hspace{0.9em}}c@{\hspace{0.9em}}|@{\hspace{0.75em}}c@{\hspace{0.75em}}|@{\hspace{0.75em}}c@{\hspace{0.75em}}|@{\hspace{0.75em}}c@{\hspace{0.75em}}|@{\hspace{0.75em}}c@{\hspace{0.75em}}|@{\hspace{0.75em}}c@{\hspace{0.75em}}|}\hline
			Type & $\varpi_{1}$ & $D(\varpi_{1})$ & $m_{\varpi_{1}}(0)$& $\varpi_{1}\cdot \rho$ & $\varpi_{1}\cdot(\varpi_{1}+2\rho)$ & $ rI(\varpi_{1})$ \\ \hline\hline
			$A_{n}$ &  $e_{1}-\tfrac{1}{n+1}\mathbbold{1}_{n+1}$ & $n+1$ & $0$ & $\tfrac{n}{2}$ & $\tfrac{n(n+2)}{n+1}$ & $1$ \\ \hline
			$B_{n}$ &  $e_{1}$ & $2n+1$ & $1$ & $n-\tfrac{1}{2}$ & $2n$ & $2$ \\ \hline
			$C_{n}$ &  $e_{1}$ & $2n$ & $0$ & $n$ & $2n+1$ & $2$ \\ \hline
			$D_{n}$ &  $e_{1}$ & $2n$ & $0$ & $n-1$ & $2n -1$ & $2$ \\ \hline
			$E_{6}$ & $\left(1,\mathbf{0}_{4},-\tfrac{1}{3},-\tfrac{1}{3},-\tfrac{1}{3}\right)$ & $27$ & $0$ & $8$ & $\tfrac{52}{3}$ & $6$ \\ \hline
			$E_{7}$ & $\left(1,\mathbf{0}_{5},-\tfrac{1}{2},-\tfrac{1}{2}\right)$ & $56$ & $0$ & $\tfrac{27}{2}$ & $\tfrac{57}{2}$ & $12$ \\ \hline
			$E_{8}$ & $\left(1,\mathbf{0}_{6},-1\right)$ & $248$ & $8$ & $29$ & $60$ & $60$ \\ \hline
			$F_{4}$ & $\left(1,\mathbf{0}_{2},-1\right)$ & $26$ & $2$ & $11$ & $24$ & $12$ \\ \hline
			$G_{2}$ & $\left(1,0,-1\right)$ & $7$ & $1$ & $5$ & $12$ & $6$ \\ \hline
	\end{tabular}}
	\caption{Summary of  $V(\varpi_{1})$ data for simple Lie algebras.
		$\mathbf{0}_{m} $ denotes the $m$-tuple of zeros.} \label{tab:pomega1_data}
\end{table}

We now describe $G_{\mathfrak{g}}^{\varpi_{1}}(\lambda,z)$ for the classical Lie algebras.
Let $\ell_{i}$ denote the $i^{\text{th}}$ Euclidean coordinate of the vector $\ell=\lambda+\rho$ for our choice of embedding of $\mathfrak{h}^{*}$. We also define  $\ell_{-i}:=-\ell_{i}$ for $1\leq i\leq n$ for $\mathfrak{g}$ of type $B_{n},C_{n}$ and $D_{n}$ and  $\ell_{0}:=0$ for $\mathfrak{g}$ of type $B_{n}$. 
\begin{corollary}\label{cor:GenPomega1}
For $\mathfrak{g}$ of classical type $A_{n}$, $B_{n}$, $C_{n}$ and $D_{n}$, then  $G_{\mathfrak{g}}^{\varpi_{1}}(\lambda,z)$ is given by
\begin{align*}
	G_{\mathfrak{sl}_{n+1}}^{\varpi_{1}}(\lambda,z) &= \dfrac{1}{z}\left(\prod_{i=1}^{n+1}\left(1+\dfrac{z}{1-(\ell_{i}-\tfrac{n}{2})z}\right)-1\right),\\
	\text{ }\\
	G_{\mathfrak{so}_{2n+1}}^{\varpi_{1}}(\lambda,z) &= \dfrac{2}{z}\left(1-\dfrac{z}{4+2nz}\right)\left(\prod_{i=-n}^{n}\left(1+\dfrac{z}{2-(\ell_{i}-n+\tfrac{1}{2})z}\right)-1\right),\\
	\text{ }\\
	G_{\mathfrak{sp}_{2n}}^{\varpi_{1}}(\lambda,z) &= \dfrac{2}{z}\left(1+\dfrac{z}{4+(2n+1)z}\right)\left(\sideset{}{'}\prod_{i=-n}^{n}\left(1+\dfrac{z}{2-(\ell_{i}-n)z}\right)-1\right),\\
	\text{ }\\
	G_{\mathfrak{so}_{2n}}^{\varpi_{1}}(\lambda,z) &= \dfrac{2}{z}\left(1-\dfrac{z}{4+(2n-1)z}\right)\left(\sideset{}{'}\prod_{i=-n}^{ n}\left(1+\dfrac{z}{2-(\ell_{i}-n+1)z}\right)-1\right),
	\end{align*}
	where the primed products exclude the  $i=0$ index.
\end{corollary}
\begin{proof}
To simplify our calculations, we initially assume that $\lambda$ is chosen so that $\ell_{- n},\ldots,\ell_{n}$ and $-\tfrac{1}{2}$ are all distinct. We will return to any exceptions to this restriction at the end of the proof. We now consider the four classical cases for $\mathfrak{g}$ in the order $A_{n}$, $D_{n}$, $C_{n}$ and $B_{n}$, since the solution methods will build on previous ones. We refer to Tables~\ref{tab:EucEmbed}, \ref{tab:g_data} and \ref{tab:pomega1_data} for relevant data.

\textbf{Case 1: $\mathfrak{g}=\mathfrak{sl}_{n+1}$.} 
For our choice of embedding,  in Table~\ref{tab:EucEmbed}, we have
\begin{align*}
	\Delta^{+}&=\lbrace e_{i}- e_{j}\hspace{0.5em}\vert\hspace{0.5em}1\leq i<j\leq n+1\rbrace,
	\\
	\Lambda(\varpi_{1})&=\lbrace e_{i}-\tfrac{1}{n+1}\mathbbold{1}_{n+1}\hspace{0.5em}\vert\hspace{0.5em}1\leq i\leq n+1\rbrace.
\end{align*}
It follows that $\ell\cdot\left(e_{i}-\tfrac{1}{n+1}\mathbbold{1}_{n+1}\right)=\ell_{i}$ as weights are embedded in the subspace of $\R^{n+1}$ orthogonal to $\mathbbold{1}_{n+1}$ and $\ell\cdot\alpha$ is of the form $\ell_{i}-\ell_{j}$  for any $\alpha\in\Delta^{+}$. This combined with Corollary~\ref{cor:GenWkappa} gives
\begin{align*}
	zG_{\mathfrak{sl}_{n+1}}^{\varpi_{1}}(\lambda,z) &
	=
	\sum_{i=1}^{n+1}\dfrac{1}{\zeta-\ell_{i}}\prod_{j=1,j\neq i}^{n+1}\dfrac{\ell_{i}+1-\ell_{j}}{\ell_{i}-\ell_{j}}
	=
	\sum_{i=1}^{n+1}\dfrac{1}{\zeta-\ell_{i}}
	\underset{\hspace{0.2em}\zeta=\ell_{i}}{\Res}\;f(\zeta),
%	\Res_{\zeta=\ell_{i}}
%	\left(\prod_{j= 1}^{n+1}\left(1+\dfrac{1}{\zeta-\ell_{j}}\right)-1\right).
\end{align*}
for 
\begin{align*}
	f(\zeta)=	\prod_{j= 1}^{n+1}\left(1+\dfrac{1}{\zeta-\ell_{j}}\right)-1,
\end{align*}
because $zG_{\mathfrak{sl}_{n+1}}^{\varpi_{1}}(\lambda,z)$ has distinct simple poles at $\zeta=\ell_{i}$ for $1\le i \le n+1$. 	 
Since $\lim_{\vert\zeta\vert\rightarrow\infty} f(\zeta)=0$,  
Cauchy's residue theorem implies $G_{\mathfrak{sl}_{n+1}}^{\varpi_{1}}(\lambda,z) =\tfrac{1}{z}f(\zeta)$. From Table~\ref{tab:pomega1_data} we find $\zeta=\tfrac{1}{z}+\tfrac{n}{2}$ so that 
\begin{align*}
	G_{\mathfrak{sl}_{n+1}}^{\varpi_{1}}(\lambda,z)= \dfrac{1}{z}\left(\prod_{i=1}^{n+1}\left(1+\dfrac{z}{1-(\ell_{i}-\tfrac{n}{2})z}\right)-1\right).
\end{align*}

For Cases 2--4 it is useful to define
\begin{align}\label{eq:Rn}
	R_{n}(\ell,\zeta):=\sideset{}{'}\prod_{j=\shortminus n}^{ n}\left(1+\dfrac{1}{\zeta-\ell_{j}}\right)
	=\prod_{j=1}^{ n}\frac{(\zeta+1)^{2}-\ell_{j}^{2}}{\zeta^{2}-\ell_{j}^{2}},
\end{align}
where the primed product excludes the $j=0$ index. By assumption, $R_{n}(\ell,\zeta)$ has $2n$ distinct simple poles at $\zeta=\ell_{\pm i}\neq -\tfrac{1}{2}$ for $1\le i\le n$. Furthermore, 
\begin{align*}
	\lim_{\vert\zeta\vert\rightarrow\infty}R_{n}(\ell,\zeta)=R_{n}\left(\ell,-\tfrac{1}{2}\right)=1.	
\end{align*}

\textbf{Case 2: $\mathfrak{g}=\mathfrak{so}_{2n}$.} For our chosen embedding 
\begin{align*}
	\Delta^{+}  &=\left\lbrace e_{i}\pm e_{j}\hspace{0.5em}\vert\hspace{0.5em} 1\leq i<j\leq n\right\rbrace,
	\\
	\Lambda(\varpi_{1})&=\left\lbrace\pm e_{i}\hspace{0.5em}\vert\hspace{0.5em} 1\leq i\leq n\right\rbrace.
\end{align*}
Hence $\ell\cdot \left(\pm e_{i}\right)=\ell_{\pm i}$ and $\ell\cdot\alpha$ is of the form $\ell_{i}\pm\ell_{j}$ for any $\alpha\in\Delta^{+}$.
 Corollary~\ref{cor:GenWkappa} gives us the expression
\begin{align*}
	\frac{z}{2}G_{\mathfrak{so}_{2n}}^{\varpi_{1}}(\lambda,z)
	 = 
	\sideset{}{'}\sum_{i=\shortminus n}^{n}\dfrac{1}{\zeta-\ell_{i}}
	\sideset{}{'}\prod_{\stackrel{j=\shortminus n}{j\neq \shortminus i,i}}^{n} 
	\dfrac{\ell_{i}+1-\ell_{j}}{\ell_{i}-\ell_{j}}
	= 
	\sideset{}{'}\sum_{i=\shortminus n}^{n}\dfrac{1}{\zeta-\ell_{i}}\underset{\hspace{0.2em}\zeta=\ell_{i}}{\Res}\;f(\zeta),
\end{align*}
for 
\begin{align}\label{eq:f_so2n}
f(\zeta)=\left(1-\dfrac{1}{2\zeta+1}\right)\left(R_{n}(\ell,\zeta)-1\right),
\end{align}
because $\tfrac{z}{2}G_{\mathfrak{so}_{2n}}^{\varpi_{1}}(\lambda,z)$ has $2n$ distinct simple poles at $\zeta =\ell_{\pm i}$ for $1\le i \le n$.
Since $f(\zeta)$ is regular at $\zeta=-\tfrac{1}{2}$  
and 
$\lim_{\vert\zeta\vert\rightarrow\infty}f(\zeta)=0$ we have  $\tfrac{z}{2}G_{\mathfrak{so}_{2n}}^{\varpi_{1}}(\lambda,z)=f(\zeta)$, by Cauchy's residue theorem. From Table~\ref{tab:pomega1_data} we find $\zeta=\tfrac{2}{z}+n-1$ so that
\begin{align*}
	G_{\mathfrak{so}_{2n}}^{\varpi_{1}}(\lambda,z) = \dfrac{2}{z}\left(1-\dfrac{z}{4+(2n-1)z}\right)\left(\sideset{}{'}\prod_{i=\shortminus n}^{n}\left(1+\dfrac{z}{2-(\ell_{i}-n+1)z}\right)-1\right).
\end{align*}

\textbf{Case 3: $\mathfrak{g}=\mathfrak{sp}_{2n}$.}  For our chosen embedding we have
\begin{align*}
	\Delta^{+} & =\left\lbrace e_{i}\pm e_{j},\hspace{0.35em}2e_{k}\hspace{0.5em}\vert\hspace{0.5em} 1\leq i<j\leq n,\hspace{0.35em}1\leq k\leq n\right\rbrace,
	\\
	\Lambda(\varpi_{1})&=\left\lbrace\pm e_{i}\hspace{0.5em}\vert\hspace{0.5em} 1\leq i\leq n\right\rbrace.
\end{align*}
Thus $\ell\cdot \left(\pm e_{i}\right)=\ell_{\pm i}$ and $\ell\cdot\alpha$ is either of the form $\ell_{i}\pm\ell_{j}$ or $2\ell_{k}$  for any $\alpha\in\Delta^{+}$. Then Corollary~\ref{cor:GenWkappa} gives us the expression
\begin{align*}
	\frac{z}{2}G_{\mathfrak{sp}_{2n}}^{\varpi_{1}}(\lambda,z)
	=
	\sideset{}{'}\sum_{i=\shortminus n}^{n}\dfrac{1}{\zeta-\ell_{i}}\left(\dfrac{\ell_{i}+1}{\ell_{i}}
	\sideset{}{'}\prod_{\stackrel{j=\shortminus n}{j\neq \shortminus i,i}}^{n}\dfrac{\ell_{i}+1-\ell_{j}}{\ell_{i}-\ell_{j}}\right)
	= 
	\sideset{}{'}\sum_{i=\shortminus n}^{n}\dfrac{1}{\zeta-\ell_{i}}\underset{\hspace{0.2em}\zeta=\ell_{i}}{\Res}\;f(\zeta),
\end{align*}
for 
\begin{align*}
f(\zeta)=	
\left(1+\dfrac{1}{2\zeta+1}\right)\left(R_{n}(\ell,\zeta)-1\right).
\end{align*}
Once again, $f(\zeta)$ is regular at $\zeta=-\tfrac{1}{2}$ and  $\lim_{\vert\zeta\vert\rightarrow\infty}f(\zeta)=0$ so by Cauchy's residue theorem, $G_{\mathfrak{sp}_{2n}}^{\varpi_{1}}(\lambda,z)=\tfrac{2}{z}f(\zeta)$. From Table~\ref{tab:pomega1_data} we find $\zeta=\tfrac{2}{z}+n$ so that
\begin{align*}
	G_{\mathfrak{sp}_{2n}}^{\varpi_{1}}(\lambda,z) = \dfrac{2}{z}\left(1+\dfrac{z}{4+(2n+1)z}\right)\left(\sideset{}{'}\prod_{i=\shortminus n}^{n}\left(1+\dfrac{z}{2-(\ell_{i}-n)z}\right)-1\right).
\end{align*}

\textbf{Case 4: $\mathfrak{g}=\mathfrak{so}_{2n+1}$.} For our chosen embedding we have
\begin{align*}
	\Delta^{+} & =\left\lbrace e_{i}\pm e_{j},\hspace{0.35em}e_{k}\hspace{0.5em}\vert\hspace{0.5em} 1\leq i<j\leq n,\hspace{0.35em}1\leq k\leq n\right\rbrace,\\
	\Lambda(\varpi_{1})&=\left\lbrace\pm e_{i}\hspace{0.5em}\vert\hspace{0.5em} 1\leq i\leq n\right\rbrace\cup\left\lbrace0\right\rbrace.
\end{align*}
Thus  $\ell\cdot \left(\pm e_{i}\right)=\ell_{\pm i}$,  $\ell\cdot0=0=\ell_{0}$ and  $\ell\cdot\alpha$ is either of the form $\ell_{i}\pm\ell_{j}$ or $\ell_{k}$  for any $\alpha\in\Delta^{+}$. Applying Corollary~\ref{cor:GenWkappa} and noting  that $\tfrac{1}{2}\varpi_{1}^{2}=\tfrac{1}{2}$ and $m_{\varpi_{1}}(0)=1$, we obtain an expression similar  to Case~3 but with an adjustment accounting for the $0$ weight. We therefore find that
\begin{align*}
	\frac{z}{2}G_{\mathfrak{so}_{2n+1}}^{\varpi_{1}}(\lambda,z)-\dfrac{1}{\zeta+\tfrac{1}{2}}
	&=
	\sideset{}{'}\sum_{i=\shortminus n}^{n}\dfrac{1}{\zeta-\ell_{i}}\left(\dfrac{\ell_{i}+1}{\ell_{i}}
	\sideset{}{'}\prod_{\stackrel{j=\shortminus n}{j\neq \shortminus i,i}}^{n}\dfrac{\ell_{i}+1-\ell_{j}}{\ell_{i}-\ell_{j}}\right)
	=\sideset{}{'}\sum_{i=\shortminus n}^{n}\dfrac{1}{\zeta-\ell_{i}}\underset{\hspace{0.2em}\zeta=\ell_{i}}{\Res}\;h(\zeta),
\end{align*}
for 
\begin{align}\label{eq:h_zeta}
	h(\zeta)=\left(1+\dfrac{1}{2\zeta+1}\right)\left(R_{n}(\ell,\zeta)-1\right).
\end{align}
Then $h(\zeta)$ is regular at $\zeta=-\tfrac{1}{2}$ and  $\lim_{\vert\zeta\vert\rightarrow\infty}h(\zeta)=0$ so by Cauchy's residue theorem, we find $\tfrac{z}{2}G_{\mathfrak{so}_{2n+1}}^{\varpi_{1}}(\lambda,z) = f(\zeta)$ where
\begin{align*}
f(\zeta)=&\dfrac{1}{\zeta+\tfrac{1}{2}}+h(\zeta)= \dfrac{1}{\zeta+\tfrac{1}{2}}\left( 1 + (\zeta+1)\left(\sideset{}{'}\prod_{j=\shortminus n}^{n}\left(1+\dfrac{1}{\zeta-\ell_{j}}\right)-1\right)\right)
	\\
	=& \dfrac{1}{\zeta+\tfrac{1}{2}}\left(\zeta \prod_{j=\shortminus n}^{n}\left(1+\dfrac{1}{\zeta-\ell_{j}}\right) -\zeta  \right)
	\\
	=&\left(1-\dfrac{1}{2\zeta+1}\right)\left(\prod_{j=\shortminus n}^{n}\left(1+\dfrac{1}{\zeta-\ell_{j}}\right)-1\right),
\end{align*}
where the unprimed product includes the $j=0$ index.
 $f(\zeta)$ has $2n+1$ simple poles at $\zeta=-\tfrac{1}{2}$ and $\zeta=\ell_{i}$ for $i\neq 0$.
 From Table~\ref{tab:pomega1_data} we find $\zeta=\tfrac{2}{z}+n-\tfrac{1}{2}$ so that
\begin{align*}
	G_{\mathfrak{so}_{2n+1}}^{\varpi_{1}}(\lambda,z) &= \dfrac{2}{z}\left(1-\dfrac{z}{4+2nz}\right)\left(\prod_{i=\shortminus n}^{n}\left(1+\dfrac{z}{2-(\ell_{i}-n+\tfrac{1}{2})z}\right)-1\right).
\end{align*}

We now consider when $\ell_{- n},\ldots,\ell_{n}$ and $-\tfrac{1}{2}$ are not all distinct. Let $\lambda=\sum_{i=1}^{n}k_{i}\varpi_{i}\in \Lambda^{+}$ so that $k_{1},\ldots,k_{n}$ are non-negative integers.   
From Table~\ref{tab:EucEmbed} for the classical Lie algebras we find that $\lambda$ has the Euclidean components shown in Table~\ref{tab:lambda_i} e.g. \cites{Hu1,Kn}. 

\begin{table}[h!]
	{\setlength{\extrarowheight}{3pt}\begin{tabular}{|@{\hspace{1em}}c@{\hspace{1em}}|@{\hspace{1em}}l@{\hspace{1em}}|}\hline
			$\mathfrak{g}$ & $\lambda$ Components\\ \hline\hline
			$\mathfrak{sl}_{n+1}$ & $\lambda_{i}=\sum_{j=i}^{n}k_{j}$  for $1\le i \le n$   and $\lambda_{n+1}=-\sum_{j=1}^{n}jk_{j}$
			\\ \hline
			$\mathfrak{so}_{2n+1}$ & $\lambda_{i}=\sum_{j=i}^{n-1}k_{j}+\tfrac{1}{2}k_{n}$ for $1\le i \le n$
			\\ \hline
			$\mathfrak{sp}_{2n}$ & $\lambda_{i}=\sum_{j=i}^{n}k_{j}$ for $1\le i \le n$
			\\ \hline
			$\mathfrak{so}_{2n}$ & $\lambda_{i}=\sum_{j=i}^{n-2}k_{j}+\tfrac{1}{2}\left(k_{n}+k_{n-1}\right)$  for $1\le i \le n-1$ and 
			$\lambda_{n}=\tfrac{1}{2}\left(k_{n}-k_{n-1}\right)$
			\\ \hline
	\end{tabular}}
	\caption{Euclidean components of $\lambda\in\Lambda^{+}$.}\label{tab:lambda_i}
\end{table}
In particular,  we find that $\lambda_{i}\ge \lambda_{i+1}$ for $1\le i\le n-1$. Since $\rho_{i}>\rho_{i+1}$ we have $\ell_{i}>\ell_{i+1}$ for any classical Lie algebra.
For $\mathfrak{g}=\mathfrak{sl}_{n+1}$, this is sufficient to guarantee all  $\ell_{i}$ are distinct since $\ell_{n+1}<0$. For $\mathfrak{g}=\mathfrak{sp}_{2n}$, we find  $\ell_{1}>\ldots>\ell_{n}>-\tfrac{1}{2}>\ell_{\shortminus n}>\ldots >\ell_{\shortminus 1}$ since $\ell_{n}\geq 1$.  So our initial assumption automatically always holds in Cases~1 and 3.

For $\mathfrak{g}=\mathfrak{so}_{2n+1}$, if $k_{n}=0$ then $\ell_{- n},\ldots,\ell_{n}$ are distinct but $\ell_{n}=\tfrac{1}{2}$ and $\ell_{- n}=-\tfrac{1}{2}$. Repeating the analysis of Case~4, we find $	\tfrac{z}{2}G_{\mathfrak{so}_{2n+1}}^{\varpi_{1}}(\lambda,z)-\left(\zeta+\tfrac{1}{2}\right)^{-1}$  now has an additional simple pole at $\zeta=-\tfrac{1}{2}$ with residue $-1$. 
Furthermore, $h(\zeta)$ of \eqref{eq:h_zeta} is given by 
\begin{align*}
h(\zeta)=\dfrac{\zeta+1}{\zeta+\tfrac{1}{2}}\left(\dfrac{\zeta+\tfrac{3}{2}}{\zeta-\tfrac{1}{2}}
\; R_{n-1}(\ell,\zeta)
-1\right),
\end{align*}
which similarly has a simple pole at $\zeta=-\tfrac{1}{2}$ with residue $-1$. Thus the same closed expression holds for $G_{\mathfrak{so}_{2n+1}}^{\varpi_{1}}(\lambda,z)$ which, in this case, has $2n-1$ simple poles at $\zeta=\ell_{i}$ for all $i\neq 0,-n$.

For  $\mathfrak{g}=\mathfrak{so}_{2n}$, if $(k_{n-1},k_{n})=(0,0)$, then  $\ell_{n}=\ell_{- n}=0$. Thus Corollary~\ref{cor:GenWkappa} implies $\tfrac{z}{2}G_{\mathfrak{so}_{2n}}^{\varpi_{1}}(\lambda,z)$ has $2n-1$ simple poles at $\zeta=\ell_{i}$ for $i\neq \pm n$ and  at $\zeta=\ell_{\pm n}=0$ with  residue
\begin{align*}
	\underset{\hspace{0.1em}\zeta=0}{\Res}\hspace{0.2em} \frac{z}{2}G_{\mathfrak{so}_{2n}}^{\varpi_{1}}(\lambda\hspace{0.125em},\hspace{0.125em}z) 
	=2\sideset{}{'}\prod_{i=1\shortminus n}^{n\shortminus1}\left(1-\dfrac{1}{\ell_{i}}\right).
\end{align*}
From the Case~2 analysis we find that with $\ell_{n}=\ell_{- n}=0$ then $f(\zeta)$ of \eqref{eq:f_so2n} is given by
\begin{align*}
f(\zeta)=&
\dfrac{(\zeta+1)^{2}}{\zeta\left(\zeta+\tfrac{1}{2}\right)}\;R_{n-1}(\ell,\zeta)-\dfrac{\zeta}{\zeta+\tfrac{1}{2}},
\end{align*} 
with the same $\zeta=0$ residue $2R_{n-1}(\ell,0)$ as $\tfrac{z}{2}G_{\mathfrak{so}_{2n}}^{\varpi_{1}}(\lambda,z) $. Thus the same closed form holds for $G_{\mathfrak{so}_{2n}}^{\varpi_{1}}(\lambda,z)$ which, in this case, has $2n-1$ distinct simple poles at $\zeta =\ell_{ i}$ for $ i \neq \pm n$ and $\zeta=\ell_{\pm n}=0$.

Lastly, if $(k_{n-1},k_{n})=(1,0)$ or $(0,1)$  then $\ell_{n} =-\ell_{-n}= \pm\tfrac{1}{2}$.  For $\ell_{n}=-\tfrac{1}{2}$ we find from Corollary~\ref{cor:GenWkappa} that $\tfrac{z}{2}G_{\mathfrak{so}_{2n}}^{\varpi_{1}}(\lambda,z) $ has a simple pole at $\zeta=-\tfrac{1}{2}$ with residue $1$. 
%\begin{align*}
%\prod_{j=1-n}^{n-1} 
%	\dfrac{-\tfrac{1}{2}+1-\ell_{j}}{-\tfrac{1}{2}-\ell_{j}}=1.
%\end{align*}
Furthermore, $f(\zeta)$ of \eqref{eq:f_so2n} is given by
\begin{align*}
	f(\zeta) &=\dfrac{\zeta}{\zeta+\tfrac{1}{2}}\left(\dfrac{\zeta+\tfrac{3}{2}}{\zeta-\tfrac{1}{2}}\hspace{0.35em}R_{n-1}(\ell,\zeta)-1\right),
\end{align*}
which also has a simple pole at $\zeta=-\tfrac{1}{2}$ with residue $1$.  The analysis for when $\ell_{-n}=-\tfrac{1}{2}$ is precisely the same. Thus the same closed formula for $G_{\mathfrak{so}_{2n}}^{\varpi_{1}}(\lambda,z)$ holds again with $2n$ distinct simple poles at $\zeta =\ell_{\pm i}$ for $1\le i \le n$ where either $\ell_{n}=-\tfrac{1}{2}$ or $\ell_{\shortminus n}=-\tfrac{1}{2}$.
Altogether, we have shown the given closed form expressions for the classical Lie algebras hold  in all cases.
\end{proof}
The generating function expressions of Corollary~\ref{cor:GenPomega1} agree with the results of Perelomov and Popov \cite{PP} and Okubo \cite{Ok} once appropriate adjustments are made for the definition of the generalised Casimir operators and the choice of minimal irreducible representation for $A_{n}$.  Our definition of the generalised Casimir $K_{p}^{\kappa}$ of \eqref{GeneralisedCasimir} is precisely that of  Eqn~(2.13) of Okubo \cite{Ok}. Furthermore, for $i=1,\ldots,n$ we find  $\ell_{i}=l_{i}$ of \cite{Ok} for $\mathfrak{g}$ of type $B_{n},C_{n}$ and $D_{n}$ and that our Casimir eigenvalues $c_{p}^{\varpi_{1}}(\lambda)$ agree with those of Okubo. From \eqref{eq:Rn} it is clear the $R_{n}(\ell,\zeta)$ is symmetric in $\ell_{1}^{2},\ldots,\ell_{n}^{2}$. Hence the Casimir  eigenvalue $c_{p}^{\varpi_{1}}(\lambda)$ is a symmetric polynomial of degree $\lfloor \tfrac{p}{2}\rfloor$ in $\ell_{1}^{2},\ldots,\ell_{n}^{2}$, as noted by Okubo.

For $\mathfrak{g}$ of type $A_{n}$, Okubo chooses   $\kappa=\varpi_{n}=-e_{n+1}+\tfrac{1}{n+1}\mathbbold{1}_{n+1}$ of minimal dimension $n+1$ for which
\begin{align*}
	\Lambda(\varpi_{n})&=\lbrace -e_{i}+\tfrac{1}{n+1}\mathbbold{1}_{n+1}\hspace{0.5em}\vert\hspace{0.5em}1\leq i\leq n+1\rbrace.
\end{align*}
Thus $\ell\cdot \mu=-\ell_{i}$ for $\mu\in\Lambda(\varpi_{n})$. Hence our $A_{n}$ Casimir invariants $c_{p}^{\varpi_{1}}(\lambda)$ agree with those of Okubo  for $\ell_{i}=-l_{i}$  for $i=1,\ldots,n+1$. We also see directly from $G_{\mathfrak{sl}_{n+1}}^{\varpi_{1}}(\lambda,z) $ that the Casimir  eigenvalue $c_{p}^{\varpi_{1}}(\lambda)$ is a symmetric polynomial of degree $p$ in $\ell_{1},\ldots,\ell_{n+1}$.

In order to compare the   generating functions of Corollary~\ref{cor:GenPomega1} to those of Perelomov and Popov \cite{PP} it is useful to express them in a universal way. 
We firstly let  $\{j\}$ denote the product indexing set for the given generating functions e.g. for $\mathfrak{sp}_{2n}$ then $\{j\}=\{\pm 1,\ldots,\pm n\}$ of cardinality $2n$. In general, we note that  $\{j\}$ is of cardinality $D(\varpi_{1})$.  
Let\footnote{The $-rI(\varpi)$ parameter reflects another normalisation choice of the bilinear form on $\mathfrak{h}^{*}$.} $\mathfrak{z}:=-\tfrac{z}{rI(\varpi_{1})}$ so that 
$\zeta=-\mathfrak{z}^{-1}+\rho_{1}$ observing that $\kappa\cdot \rho=\rho_{1}$ for $\kappa=\varpi_{1}$ for each classical algebra. Then Corollary~\ref{cor:GenPomega1} implies
\begin{align}
	G_{\mathfrak{g}}^{\varpi_{1}}(\lambda,z)=&\left(
	1-\frac{b}{2\zeta+1} \right)
	\frac{1}{\mathfrak{z}}
	\left( 1-P (\mathfrak{z})\right)
	\notag
	\\
	\label{eq:GenUniv}
		=&\left(
	1+\frac{b\,\mathfrak{z}}{2 -(2\rho_{1}+1)\mathfrak{z}} \right)
	\frac{1}{\mathfrak{z}}
	\left( 1- P (\mathfrak{z}) \right),
\end{align}
for  
\begin{align*}
	P(\mathfrak{z}):=\prod_{\{j\}}\left(1+\frac{1}{\zeta-\ell_{j}} \right)
	=\prod_{\{j\}}\left(1-\frac{\mathfrak{z}}{1+\left(-\ell_{j}+\rho_{1}\right)\mathfrak{z}  } \right),
\end{align*}
where $b=0,1,-1,1$ for  $\mathfrak{g}$ of type $A_{n}, B_{n}, C_{n}, D_{n}$ respectively. Since $\sum_{\{j\}}\ell_{j}=0$ in all cases, we find \eqref{eq:GenUniv} implies
\begin{align*}
	G_{\mathfrak{g}}^{\varpi_{1}}(\lambda,z)=D(\varpi_{1})\left(1 +\left(\frac{b}{2} - \frac{\left(D(\varpi_{1})-1\right)}{2}+\rho_{1}\right)\mathfrak{z} +O(\mathfrak{z}^2)\right).
\end{align*}
Following \eqref{GeneralisedCasimir} we know that $G_{\mathfrak{g}}^{\varpi_{1}}(\lambda,z)=D(\varpi_{1})+O(\mathfrak{z}^2)$ so that
\begin{align*}
	b= D(\varpi_{1})-1-2\rho_{1}.
\end{align*}
It is easy to confirm that $b=0,1,-1$ or $1$ from Table~4. Comparing to Perelomov and Popov's results  we find that their parameters $\alpha$ and $\beta$ of Table~1 of \cite{PP} are in fact given by $\alpha=\rho_{1}$ and $\beta=b=D(\varpi_{1})-1-2\rho_{1}$. Finally, their generating function, Eqn~(75) of \cite{PP}, is given by \eqref{eq:GenUniv} for $\mathfrak{g}$ of type $B_{n}, C_{n}$ and $D_{n}$ (using $\ell_{-j}=-\ell_{j}$) whereas for $A_{n}$ one again requires use of the minimal fundamental representation $\varpi_{n}$.

The Perelomov and Popov results were  computed from a triagonalisation of representation matrices and the question of a more comprehensive description or what other representations could admit such a generating function description of their Casimir invariants was left to conjecture. Okubo's formula does provide a more comprehensive description, connecting the terms of the generating function to the irreducible components of a tensor representation. Further, Okubo specifically describes the eigenvalues of the generalised Casimir operators associated with the first fundamental representations of classical Lie algebras as admitting convenient formulae \cite{Ok}. Theorem~\ref{theor:CasGen} adds more context to the underlying mechanism, showing that these convenient formulae were a result of the $\lambda+\mu$ terms all being either dominant or not linked to a dominant element, in the case of $\kappa=\varpi_{1}$ for classical Lie algebras. This is effectively because for any choice of $\lambda$, the weights $\mu\in\Lambda(\varpi_{1})$ are short enough that the $\lambda+\mu$ terms cannot be far enough away from $\lambda$ to be a non-dominant element linked to a dominant one - and thus contribute an additional term to the computation as we see via Klimyk's formula.
The Euclidean embeddings chosen in Table~\ref{tab:EucEmbed} also play a crucial role for the classical Lie algebras in that $\varpi=e_{1}$ (or its projection for $\mathfrak{g}=\mathfrak{sl}_{n+1}$). Since the Weyl group $W$ acts by coordinate permutations and possible sign changes, the $W$ action on $e_{1}$ generates (projected) elements of $\{\pm e_{i}\}_{i=1}^{n}$. Hence,  for the classical algebras, 
$zG_{\mathfrak{g}}^{\varpi_{1}}(\lambda,z)$ has simple poles in $\zeta=\ell\cdot \mu\in \{\pm \ell_{i}\}_{i=1}^{n}$ for all $\mu\in W(\varpi_{1})$. This may be why no closed generating function formulas are known for the exceptional Lie algebras - at least for standard Euclidean embeddings.

%Another contributor to the expressions we see in Corollary $3.10$ is the particular choice of embedding. It is somewhat of a coincidence that under the conventional choices of embeddings for the root spaces of classical Lie algebras we see that the standard coordinates of $\ell=\lambda+\rho$ fit into concise product expressions for the generating functions of $c_{p}^{\varpi_{1}}(\lambda)$. It is possible that alternate choices of embedding may admit useful generating functions for $c_{p}^{\varpi_{1}}$ over the exceptional Lie algebras or for $c_{p}^{\varpi_{n\shortminus1}}(\lambda)$ or $c_{p}^{\varpi_{n}}(\lambda)$ over Lie algebras of type $D_{n}$.

\section{Conclusion}\label{sec:Conc}
Theorem~\ref{theor:CasGen} establishes a means of calculating generalised Casimir operators associated with $V(\kappa)$ for any finite-dimensional, simple, complex Lie algebra using only knowledge of $\Delta^{+}$ and $\Lambda(\kappa)$, which is computationally advantageous for $\Lambda(\kappa)$ small. The resulting generating functions have been shown to be equivalent to established closed form formulae as in Corollary~\ref{cor:GenPomega1}. It is known that, with the exception of Lie algebras of type $D_{n}$, that knowledge of these operators' action on $V(\lambda)$ for arbitrary $\lambda$ is sufficient to describe the action of $Z(U(\mathfrak{g}))$ on any irreducible representation of $\mathfrak{g}$ (see Table ~\ref{tab:ZU_Gens}).
We also note the development in \cite{MSV} of a universal Casimir eigenvalue generating function $G_{\mathfrak{g}}^{\ad}(\ad,z)$ for the adjoint representation $\ad$ and simple $\mathfrak{g}$ expressed in terms of Vogel parameters \cite{Vo}.  

It remains to be determined if there is a closed form expression for the generating functions  $G^{\varpi_{n-1}}_{\mathfrak{so}_{2n}}(\lambda,z)$ or $G^{\varpi_{n}}_{\mathfrak{so}_{2n}}(\lambda,z)$. As the weights of $\Lambda(\varpi_{n-1})$ and $\Lambda(\varpi_{n})$ are not distributed as neatly over the standard basis vectors in our embedding as with $\Lambda(\varpi_{1})$ there may be some utility in exploring alternate embeddings for the roots of $\mathfrak{so_{2n}}$. 
We can readily compute $G_{\mathfrak{g}}^{\varpi_{1}}(\lambda,z)$ for the exceptional Lie algebras \cite{Fl} but there are no known closed form expressions which would also be worth investigating. The first fundamental representations of $E_{8}$ and $F_{4}$ are ``non-degenerate'' in the sense that they each have weights (specifically the weight $0$) with multiplicity greater than $1$, which may interfere or prevent finding such closed form expressions as speculated by Perelomov and Popov \cite{PP}.

It is known that many of the fundamental representations of the simple Lie algebras - though not all - can be derived from the first fundamental representation of that Lie algebra as exterior products or their decompositions, having their algebraic information entirely determined by the first fundamental representation. It stands to reason then, that it may be possible to derive a relationship between the generating functions of the first fundamental representations and their dependent fundamental representations. However, the same issues with these representations being ``non-degenerate" may interfere.

It is also of note that the $\tfrac{D(\lambda+\mu)}{D(\lambda)}$ terms  - with $D(\cdot)$ as in \eqref{eq:D_defn} - which appear in our calculations for the classical Lie algebras contain parts that can be viewed as ratios of Vandermonde determinants as a consequence of the underlying structure of the roots with respect to the choice of embedding. Specifically, for Lie algebras of type $A_{n}$ we see that the ratio of products appears as
\begin{align*}
	\smashoperator[r]{\prod_{1
			\leq i<j\leq n+1}}\hspace{0.8em}\frac{(\ell+\mu)\cdot(e_{i}-e_{j})}{\ell\cdot(e_{i}-e_{j})}\,=\smashoperator[r]{\prod_{1\leq i<j\leq n+1}}\hspace{0.8em}\frac{\ell_{i}+\mu_{i}-(\ell_{j}+\mu_{j})}{\ell_{i}-\ell_{j}},
\end{align*}
while for classical Lie algebras of other types we have the following pattern within each of their respective $\tfrac{D(\lambda+\mu)}{D(\lambda)}$ terms
\begin{align*}
	\smashoperator[r]{\prod_{1\leq i<j\leq n}}\hspace{0.7em}\frac{(\ell+\mu)\cdot(e_{i}-e_{j})}{\ell\cdot(e_{i}-e_{j})}\frac{(\ell+\mu)\cdot(e_{i}+e_{j})}{\ell\cdot(e_{i}+e_{j})}\,=\smashoperator[r]{\prod_{1\leq i<j\leq n}}\hspace{0.7em}\frac{(\ell_{i}+\mu_{i})^{2}-(\ell_{j}+\mu_{j})^{2}}{\ell_{i}^{2}-\ell_{j}^{2}}.
\end{align*}
This provides another avenue for deeper inspection, either to better explicitly understand the underlying mechanics that result in these patterns or in pursuit of generating function expressions that could produce these patterns when expanded as a sum over $\Lambda(\kappa)$.

\end{document}